\documentclass[letterpaper]{article} 
\usepackage{aaai23-arXiv}  
\usepackage{times}  
\usepackage{helvet}  
\usepackage{courier}  
\usepackage[hyphens]{url}  
\usepackage{graphicx} 
\usepackage{natbib}  
\usepackage{caption} 
\title{Safeguarded Learned Convex Optimization}
\author{
    Howard Heaton,\equalcontrib\textsuperscript{\rm 1}\ 
    Xiaohan Chen,\equalcontrib\textsuperscript{\rm 2}\ 
    Zhangyang Wang,\textsuperscript{\rm 2}
    Wotao Yin\textsuperscript{\rm 3} 
}
\affiliations{
    \textsuperscript{\rm 1}Typal Research, Typal LLC\\
    \textsuperscript{\rm 2}Department of Electrical and Computer and Engineering, The University of Texas at Austin\\
    \textsuperscript{\rm 3}Alibaba US, DAMO Academy, Decision Intelligence Lab\\

}

\usepackage{bibentry}

\newcommand{\eg}{\textit{e.g.}\ }
\newcommand{\ie}{\textit{i.e.}\ }
\usepackage{amsfonts}
\newcommand{\bbR}{\mathbb{R}}
\newcommand{\bbN}{\mathbb{N}}
\newcommand{\II}{\mathrm{I}}
\newcommand{\LTO}{\mathcal{L}2\mathcal{O}}
\newcommand{\limk}{\lim_{k\rightarrow\infty}}

\usepackage{amsmath,amssymb,amsthm}

\DeclareMathOperator*{\argmin}{arg\,min}
\usepackage{braket}
\newcommand{\fix}[1]{\mathrm{fix}(#1)}
\newcommand{\prox}{\mathrm{prox}}
\newtheorem{theorem}{Theorem}
\newtheorem{assumption}{Assumption}

\newtheorem{corollary}{Corollary}
\newtheorem{lemma}{Lemma}
\newtheorem{remark}{Remark}
\usepackage{multirow, makecell}
\usepackage{subfig}
\usepackage{xcolor}

\usepackage{algpseudocode}
\usepackage{algorithm} 
\makeatletter
\renewcommand*{\ALG@name}{Algorithm}
\makeatother

\begin{document}

\maketitle

\begin{abstract}
Applications abound in which optimization problems must be repeatedly solved, each time with new (but similar) data. 
Analytic optimization algorithms can be hand-designed to provably solve these problems in an iterative fashion. 
{On one hand, data-driven algorithms can ``learn to optimize'' (L2O)  with much fewer iterations and   similar cost per iteration as general-purpose optimization algorithms. On the other hand, unfortunately, many L2O algorithms lack converge guarantees.}
To fuse the advantages of these approaches, we present a Safe-L2O framework. Safe-L2O updates incorporate a safeguard to guarantee convergence for convex problems with proximal and/or gradient oracles. The safeguard is simple and computationally cheap to implement, and it is activated only when the data-driven L2O updates would perform poorly or appear to diverge. This yields the numerical benefits of employing machine learning   to create rapid L2O algorithms while still guaranteeing convergence. Our numerical examples show    convergence of Safe-L2O algorithms, even when the provided data is \textit{not} from the distribution of training data.
\end{abstract}

Solving scientific computing problems often requires  application of efficient and scalable optimization algorithms. 
Data-driven algorithms can execute in much fewer iterations  and with   similar cost per iteration as state-of-the-art general purpose algorithms. 
Inspired by one such algorithm called ISTA, \cite{GregorLeCun2010_learning} proposed treating the entries in fixed matrices/vectors of the algorithm as learnable parameters that can vary by iteration. These entries were   fine-tuned to obtain optimal performance on a data set for a fixed number of iterations.  
Empirically, this approach converged and showed roughly a 20-fold reduction in computational cost compared to the original analytic algorithm.
Several related works followed, also demonstrating numerical success (discussed below). 
These efforts open the door to a new class of algorithms and analyses. Analytic optimization results often provide worst-case convergence rates, and limited theory exists pertaining to instances drawn from a  common distribution (\eg data supported on a low-dimensional manifold). Most L2O methods have little or no convergence guarantees, especially on data \textit{distinct} from what is seen in training. This work   addresses the inquiry:
\begin{center}  
		\it 
	    Can a safeguard be added to    L2O algorithms to give  convergence   without significantly hindering performance? 
\end{center}
A safeguard is anything that identifies when a ``bad'' L2O update would occur and what to do in place of that ``bad'' update.
{Informally, what a safeguard does can be summarized by generating a sequence $\{x^k\}$ with updates of the form}
\begin{equation}
    x^{k+1} = 
    \begin{cases}
        \begin{array}{cl}
        \mbox{L2O Update} & \mbox{if  update is ``good''} \\
        \mbox{Fallback Update} & \mbox{otherwise.}
        \end{array}
    \end{cases}
\end{equation}%
We provide an affirmative answer to the question for convex problems with gradient and/or proximal oracles by providing such a safeguard and replacing ``bad'' L2O updates with updates from analytic methods.  
Our framework is called Safe-L2O.
Since a trade-off is formed between per iteration costs and  ensuring convergence, we clarify properties constituting a ``practical'' L2O safeguard as follows.
\begin{enumerate}
    \item The safeguard should ensure certain forms of \textit{worst-case} convergence similar to analytic algorithms.

    \item  The safeguard must only use known quantities related to   convex problems
    (\eg objective values, gradient norms).
    
    \item  Both L2O and Safe-L2O schemes should perform identically on ``good'' data with comparable per-iteration costs.
    
    \item The safeguard should kick in only when ``bad'' L2O updates would otherwise occur.
\end{enumerate} 

The core challenge is to create a simple safeguard that kicks in only when needed. Unlike classic optimization algorithms, exceptional L2O algorithms do {\it not} necessarily exhibit the behavior that each successive iterate is ``better'' than the current iterate (\ie are not monotonically improving). Loosely speaking, this means there are cases where an L2O scheme that gets ``worse'' for a couple iterates yields a better final output than an L2O scheme that is required to get ``better'' at each iterate. The intuition behind why this can be acceptable is   we are   interested in the final output of the L2O algorithm and L2O schemes may learn ``shortcuts.'' 
From this insight, we deduce the safeguard should exhibit a form of trailing behavior, \ie it should measure progress of previous iterates and only require that updates are ``good'' on average. If the safeguard {triggers too often}, then the Safe-L2O scheme's flexibility and performance are limited. If it {triggers too rarely}, then the Safe-L2O scheme may exhibit highly oscillatory behavior and converge   slowly.


In addition to L2O updates, our  method uses a safeguard condition with the update formula from a conventional algorithm. 
When the ``good'' condition holds, the L2O update is used; when it fails, the formula from the conventional algorithm is used. In the ideal case, L2O updates are applied often and the conventional algorithm formula provides a ``fallback'' for exceptional cases. This fallback is designed together with the safeguard condition to ensure convergence. This also implies, even when an L2O algorithm has a fixed number of iterations with tunable parameters, the algorithm may be extended to an arbitrary number of iterations by applying the fallback to compute latter updates (see Figure \ref{fig: ALISTA-plots}).


{\bf Review of L2O Methods.}   
A seminal L2O work in the context of sparse coding was by \cite{GregorLeCun2010_learning}.
Numerous follow-up papers also demonstrated empirical success at constructing rapid regressors approximating iterative sparse solvers  for compression, nonnegative matrix factorization, compressive sensing and other applications \cite{SprechmannBronsteinSapiro2015_learning,WangLingHuang2016_learning,WangLiuChangLingYangHuang2016_d3,HersheyRouxWeninger2014_deep,YangSunLiXu2016_deep}. A summary of unfolded optimization procedures for sparse recovery is given by \cite{AblinMoreauMassiasGramfort2019_learning} in Table A.1. The majority of L2O works pertain to sparse coding and provide limited theoretical results. 
Some works have interpreted LISTA in various ways to provide proofs of different convergence properties \cite{GiryesEldarBronsteinSapiro2018_tradeoffs,MoreauBruna2017_understanding}. Others have investigated structures related to LISTA \cite{XinWangGaoWipfWang2016_maximal,BlumensathDavies2009_iterative,BorgerdingSchniterRangan2017_ampinspired,MetzlerMousaviBaraniuk2017_learned}, providing  results varying by assumptions.
 \cite{ChenLiuWangYin2018_theoretical} introduced necessary
conditions for the LISTA weight structure to asymptotically achieve a linear convergence rate. This was followed by \cite{LiuChenWangYin2019_alista}, which further simplified the weight conditions and provided a result stating that, with high probability, the convergence rate of LISTA is at most linear.
The mentioned results are useful, yet can require intricate assumptions and proofs specific to the sparse coding problems.
We refer readers to the recent survey \cite{chen2022learning} for a more comprehensive overview of L2O methods.


Our safeguarding scheme is related to existing works in Krasnosel'ski\u{\i}-Mann (KM) methods. 
The SuperMann scheme \cite{ThemelisPatrinos2019_supermann} presents a KM method that safeguards in a more hierarchical manner than ours and solely refers to the current iterate residuals (plus a summable sequence). Additionally, a similar safeguarding setup has been used for Anderson accelerated KM methods  \cite{ZhangODonoghueBoyd2018_Globally}. These methods are not designed with L2O in mind and differ from our approach both in the assumptions used and particular update formulae.

{\bf Our Contribution.} 
We provide a simple framework, Safe-L2O, for wrapping data-driven algorithms with convergence guarantees.  This framework can be used with all L2O algorithms that solve convex problems for which proximal and/or gradient oracles are available. We incorporate multiple safeguarding procedures in a general setting and present a simple procedure for utilizing machine learning methods to instill knowledge from available data.  These results   form a single, general framework for use by practitioners. 

\section{Fixed Point Methods}

Fixed-point iteration abstracts most of the convex optimization methods that are based on gradient and/or proximal oracles, and those methods are targets of recent L2O accelerations. Our method is based on examining the fixed-point residual, and so it applies to a wide variety of L2O methods, even   complicated first-order methods (\eg ADMM and primal-dual methods) where the underlying problems can include constraints.
To provide a brief background, here we overview fixed point methods.
Denote the set of fixed points of each operator $T\colon \bbR^n\rightarrow\bbR^n$ by $\fix{T} \triangleq \{x: Tx = x\}.$
For an operator $T$ with a nonempty fixed point set (\ie $\fix{T}\neq\emptyset$), we consider the fixed point problem 
\begin{equation}
\mbox{Find $x^\star$ such that $x^\star \in \fix{T}$.}
\label{eq: fixed-point-problem}
\end{equation}
See Table \ref{table: NE-operators} for examples of the operator $T$ in convex minimization.
We focus on  fixed point iteration to give a general approach for creating  sequences that converge to solutions of (\ref{eq: fixed-point-problem}) and, thus, of the corresponding optimization problem.  
 
The following definitions are used throughout. 
An operator $T\colon \bbR^n\rightarrow\bbR^n$ is \textit{ nonexpansive} if it is 1-Lipschitz,\footnote{The Euclidean norm on $\bbR^n$ is denoted by $\|\cdot\|$.} \ie 
\begin{equation}
    \|T(x)-T(y)\| \leq \|x-y\|, \ \ \ \mbox{for all $x,y\in\bbR^n$.}
\end{equation}
An operator $T$ is {\it  averaged} if there exists $\vartheta \in (0,1)$ and  a nonexpansive operator $Q :\bbR^n\rightarrow \bbR^n$ such that $T = (1-\vartheta) \II + \vartheta Q$, with $\II$ the identity.
%
\noindent A classic theorem states   sequences generated by successively applying an averaged operator converge to a fixed point. 
In all of this work, we assume each operator $T$ is averaged.
This method comes from \cite{Krasnoselskii1955_two} and \cite{Mann1953_mean}, which yielded adoption of the name {\it Krasnosel'ski\u{\i}-Mann} (KM) method. This result is stated below.

\begin{theorem}    \label{thm: KM-convergence}
	If an averaged operator $T:\bbR^n\rightarrow\bbR^n$ has a nonempty fixed point set and a sequence $\{x^k\}$ with arbitrary initial iterate $x^1\in\bbR^n$ satisfies the update relation
	\begin{equation}
	x^{k+1} = T(x^k), \ \ \ \mbox{for all $k\in\bbN$,}
	\label{eq: classic-KM-iteration}
	\end{equation}
	then there is a solution $x^\star \in \fix{T}$ to (\ref{eq: fixed-point-problem}) such that the sequence $\{x^k\}$ converges to $x^\star$.
\end{theorem}

\begin{table*}[t]
    \caption{Averaged operators for well-known algorithms. We assume $\alpha > 0$ and, when $\alpha$ is multiplied by a gradient, we also assume $\alpha < 2/L$, with $L$ the Lipschitz constant for the gradient. The dual of a function is denoted by a superscript $*$, and $\Omega = \{ (x,z) : Ax+Bz=b\}$. Operators $J$ and $R$ are defined in equations (\ref{eq: resolvent}) and (\ref{eq: reflected-resolvent}), respectively. The block matrix $M$ is $M=[\alpha^{-1}\mbox{Id}, A^T; -A,\beta^{-1}\mbox{Id}]$. In each case, $\mathcal{L}$ is the Lagrangian associated with the presented problem.}	 
	\label{table: NE-operators}	
	    \centering
	    {\small 
		\begin{tabular}{c|c|c} 
			Problem & Method & Fallback Operator $T$  
			\\[4.5 pt] \hline\hline &  & \\[-6 pt]
			$\min f(x)$ & Gradient Descent & $\mbox{Id}-\alpha \nabla f$  
	    	\\[4.5 pt] \hline &  & \\[-6 pt]
			$\min f(x)$ & Proximal Point & $\prox{\alpha f}$  
			\\[4.5 pt] \hline &  & \\[-6 pt]
			$\min\{ g(x) : x \in C\}$ & Projected Gradient & $\mbox{proj}_C\circ \left( \mbox{Id} - \alpha \nabla g \right)$ 
			\\[4.5 pt] \hline &  & \\[-6 pt]			
			$\min f(x) + g(x)$ & Proximal Gradient & $\prox{\alpha f}\circ\left( \mbox{Id} - \alpha \nabla g \right)$ 
		    \\[4.5 pt] \hline &  & \\[-6 pt]
			$\min f(x) + g(x)$ & Douglas-Rachford & $\frac{1}{2}\left(\mbox{Id}  + R_{\alpha\partial f}\circ R_{\alpha \partial g}\right)$  
			\\[4.5 pt] \hline &  & \\[-6 pt]		
			$\displaystyle \min_{(x,z)\in\Omega}  f(x) + g(z)$ 
			& {ADMM}& $\frac{1}{2}\left(  \mbox{Id} + R_{\alpha A \partial f^* (A^T\cdot )} \circ R_{\alpha (B \partial g^*(B^T\cdot)-b)}  \right)$  
			\\[4.5 pt] \hline &  & \\[-6 pt]
			$\mbox{min} f(x) \ \mbox{s.t.} \ Ax=b$
			& Uzawa
			& $\mbox{Id} + \alpha \left(A\nabla f^*(-A^T\cdot )-b\right)$ 
			\\[4.5 pt] \hline &  & \\[-6 pt]
             {$\mbox{min} f(x) \ \mbox{s.t.} \ Ax=b$}
			& Proximal Method of Multipliers
			&  {$J_{\alpha \partial \mathcal{L}}$ } 
			\\[4.5 pt] \hline &  & \\[-6 pt]			
			$\mbox{min} f(x) + g(Ax)$
			& Primal-Dual Hybrid Gradient
			& $J_{M^{-1}\partial \mathcal{L}}$
	\end{tabular}}
\end{table*}

For reference, we briefly discuss the use of resolvents in conventional algorithms.
Consider a convex function $f\colon\bbR^n\rightarrow\bbR$ with subgradient $\partial f$.
For $\alpha > 0$, the {\it resolvent} $J_{\alpha\partial f}$ of  $\alpha \partial f$ is  defined by $J_{\alpha \partial f}(x) 
    \triangleq (\mbox{I} + \alpha \partial f)^{-1}(x)$, \ie 
    \begin{align}
    J_{\alpha \partial f}(x)  
    = \left\lbrace y :  (x-y)/\alpha\in \partial f(y) \right\rbrace
	\label{eq: resolvent}
    \end{align}    
and the {\it reflected resolvent} of $\partial f$ is
\begin{equation}
    {R_{\alpha \partial f}}(x) \triangleq   (2J_{\alpha \partial f} - \mbox{Id})(x)
    = 2J_{\alpha \partial f}(x) - x.
    \label{eq: reflected-resolvent}
\end{equation}
If $f$ is closed, convex, and proper, then the resolvent is precisely the proximal operator, \ie 
\begin{equation}
J_{\alpha \partial f}(x)
=\mbox{prox}_{\alpha f}(x)
 \triangleq \argmin_{z\in\bbR^n} \alpha f(z) + \dfrac{1}{2}\|z-x\|^2.
\end{equation} 
Proximal operators for several well-known functions can be expressed by explicit formulas (\eg see page 177 in \cite{Beck2017_First}).
It can be shown that $R_{\alpha \partial f}$ is nonexpansive and $J_{\alpha \partial f}$ is averaged   \cite{BauschkeCombettes2017_convex}. 
Table \ref{table: NE-operators} provides several examples of these operators  in well-known optimization algorithms. 

\section{Safeguarded L2O Method}    
\label{sec: Safe-L2O}

\begin{algorithm}
\caption{ L2O Network\  (No Safeguard)}
\label{alg: L2O_Abstract}
\begin{algorithmic}[1]           
    \State $\LTO(d;\ \Theta):$ 
    
    \State{\begin{tabular}{p{0.2\textwidth}r}
     \hspace*{0pt} $x^1 \leftarrow \tilde{x}$
     & 
     {\footnotesize $\vartriangleleft$ Initialize inference}
     \end{tabular}}        

    \State{\begin{tabular}{p{0.2\textwidth}r}
     \hspace*{0pt} {\bf for } $k=1,2,\ldots,K$
     & 
     $\vartriangleleft$  Loop for each layer
     \end{tabular}}  
     
    \State{\begin{tabular}{p{0.2\textwidth}r}
     \hspace*{8pt} $x^{k+1} \leftarrow T_{\Theta^k}(x^k;d)$
     & 
     {\footnotesize $\vartriangleleft$ L2O Update }
     \end{tabular}}      


    \State{\begin{tabular}{p{0.2\textwidth}r}
     \hspace*{0pt} {\bf return} $x^{K+1}$
     & 
     {\footnotesize $\vartriangleleft$ Output inference}
     \end{tabular}}   
\end{algorithmic}
\end{algorithm}

This section presents the  Safe-L2O framework.
The safeguard acts as a wrapper around a data-driven algorithm, which is formulated in practice as a neural network.
Each L2O operator, denoted throughout by $T_{\Theta^k}$, is parameterized by layerwise weights $\Theta = (\Theta^1, ..., \Theta^K)$. Input data $d$ is used to define an optimization problem   (\eg the measurement vector in a least squares problem). To make this dependence clear, we   often write $T(\cdot;\ d)$. Often $T_{\Theta^k}(\cdot ;d)$ can be viewed as forming one or multiple layers of a feed forward network. Thus,  we interpret $\mathcal{N}_\Theta(d)$ in Algorithm \ref{alg: L2O_Abstract}      as a feed forward network. In addition to an L2O operator $T_{\Theta^k}$, our Safe-L2O method uses a fallback operator $T$ {(unrelated to $T_{\Theta^k}$)} and a scalar sequence $\{\mu_k\}$. Here $T$ defines an averaged operator from the update formula of a conventional optimization algorithm. Each $\mu_k$ defines a reference value to determine whether a tentative L2O update is ``good.'' 
Each reference value $\mu_k$ in our   safeguarding schemes is related to a combination of $\|y^i - T(y^i;d)\|$ and $\|y^i - x^i\|$ among previous iterates $i=1,\dots,k$, where $y^i = T_{\Theta^i}(x^i;d)$.


\begin{algorithm}
\caption{ Safeguarded L2O (Safe-L2O)}
\label{alg: Safe-L2O}
\begin{algorithmic}[1]           
    \State $\text{Safe-L2O}(\LTO(d;\Theta), T, \alpha,\ \beta )$
    
    \State{\begin{tabular}{p{0.47\textwidth}r}
     \hspace*{0pt}  
     $x^0 \leftarrow \tilde{x}$, \quad 
     $x^1 \leftarrow \tilde{x}$,\quad 
     $y^1 \leftarrow T_{\Theta^1}(x^1)$,\quad 
     $k\leftarrow 1$
     & 
     \end{tabular}}        

    \State{\begin{tabular}{p{0.47\textwidth}r}
     \hspace*{0pt}  $\mu_1 \leftarrow \alpha^{-1}\cdot (\|y^1 - T(y^1;d)\| + \beta \|y^1 - x^1\|)$
     & 
     \end{tabular}}          

    \State{\begin{tabular}{p{0.47\textwidth}r}
     \hspace*{0pt} {\bf while} $\|x^{k}-x^{k-1}\| > \varepsilon$ {\bf or} $k =1$
     & 
     \end{tabular}}  
     
    \State{\begin{tabular}{p{0.47\textwidth}r}
     \hspace*{8pt} $y^k \leftarrow T_{\Theta^k}(x^k;d)$
     & 
     \end{tabular}}      
     
    \State{\begin{tabular}{p{0.47\textwidth}r}
     \hspace*{8pt} {\bf if} $\| y^k - T(y^k;d) \| + \beta \|y^k-x^k\|\leq \alpha \mu_k$
     & 
     \end{tabular}}      
     
    \State{\begin{tabular}{p{0.47\textwidth}r}
     \hspace*{16pt} $x^{k+1} \leftarrow y^k$
     & 
     \end{tabular}}

    \State{\begin{tabular}{p{0.47\textwidth}r}
     \hspace*{8pt} {\bf else}
     & 
     \end{tabular}}      
     
    \State{\begin{tabular}{p{0.47\textwidth}r}
     \hspace*{16pt} $x^{k+1} \leftarrow  T(x^k;d)$
     & 
     \end{tabular}}

    \State{\begin{tabular}{p{0.47\textwidth}r}
     \hspace*{8pt} Update safeguard $\mu_{k+1}$
     & 
     \end{tabular}}     

    \State{\begin{tabular}{p{0.47\textwidth}r}
     \hspace*{8pt} $k\leftarrow k+1$
     & 
     \end{tabular}}       

    \State{\begin{tabular}{p{0.47\textwidth}r}
     \hspace*{0pt} {\bf return} $x^k$
     & 
     \end{tabular}}   
\end{algorithmic}
\end{algorithm}  

 We propose the Safe-L2O scheme in Algorithm~\ref{alg: Safe-L2O}.
{As   shown in Line 1, a safeguarded L2O operator consists of an L2O network $\LTO(d;\Theta)$, a fallback operator $T$, and a parameter  $\alpha\in (0,1)$.
Here $\Theta = (\Theta^1, \ldots, \Theta^K)$ forms layerwise weights $\Theta^k$ that define the L2O update $T_{\Theta^k}$ at the $k$-th iteration.
These parameters are trained beforehand, following standard training methods for machine learning models (as outlined in the Appendices).
Table~\ref{table: NE-operators} shows numerous choices of the fallback operator $T$ for different combinations of optimization problems and algorithms
In Line 2, the initial iterate $x^1$ is chosen to be an arbitrary (but fixed) vector $\tilde{x}$.
The initial iterate $\mu_1$ of the safeguard sequence $\{\mu_k\}$ is   initialized using the initial iterate $x^1$, an L2O update $y^1$, and the fallback operator $T$ in Line 3.
%
From Line 4 to Line 11, a repeated loop occurs to compute each update $x^{k+1}$. In Line 5 the L2O operator is applied to the current iterate $x^k$ get a tentative update $y^k$. 
This $y^k$ is  ``good'' if the the inequality in Line 6 holds. In such a case, the L2O update is assigned to $x^{k+1}$ in Line 7. Otherwise, the fallback operator $T$ is triggered to obtain the update $x^{k+1}$ in Line 9.
Note the initial iterate $\mu_1$ is defined such that $x^2 = y^1$, \ie the first L2O update is always accepted.
Lastly, the safeguard parameter is updated in Line 10. Refer to Table~\ref{table: update-mu} for choices of schemes to update the safeguard parameters.} 

\begin{figure*}[t]
    \centering
    \subfloat[Performance on {\it seen} distribution]{\includegraphics[height=2.0in]{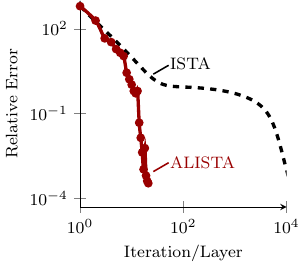}}
    \subfloat[Performance on {\it unseen} distribution]{\includegraphics[height=2.2in]{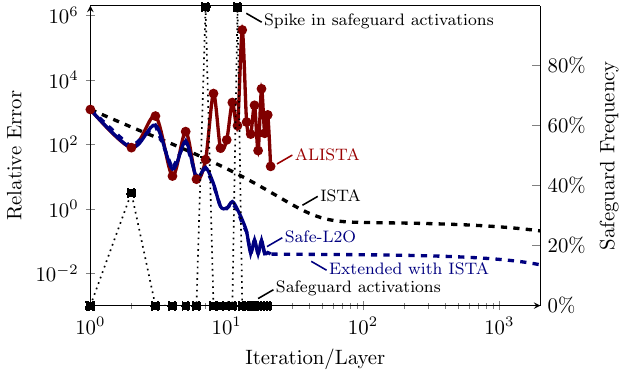}}
    \caption{Plot of error versus iteration for ALISTA example.
     {Here ISTA is the classic algorithm, ALISTA is the L2O operator in Algorithm~\ref{alg: L2O_Abstract} and Safe-L2O is the safeguarded version of ALISTA in Algorithm~\ref{alg: Safe-L2O}.}
	Trained with $\phi_d = f_d$. Inferences used $\alpha = 0.99$ and $\mbox{EMA}(0.25)$. 
    In (b), how often the L2O update is ``bad'' and the  safeguard activates for Safe-L2O is indicated in reference to the right vertical axis. This plot shows the safeguard is used only when $k=2$, $k=7$, and $k=12$.}
    \label{fig: ALISTA-plots}
\end{figure*}

   Below are standard   assumptions  used to prove the main result. The first enables a fixed point formulation.
   
\begin{assumption}\label{ass: T}
	For input data $d$, the   optimization problem has a solution and there is an operator $T$ such that 
	i) $\fix{T(\cdot;d)}$ is the solution set  and
	ii) $T(\cdot ;d)$ is averaged.
\end{assumption}

\noindent The next assumption ensures   used L2O updates approach  solutions. This is done by computing the safeguard value, which is a fixed point residual with the fallback operator.
\begin{assumption} \label{ass: mu-k}
    Here $\alpha \in [0,1)$, $\beta \in (0,\infty)$ and the safeguard   $\{\mu_k\}$ is monotonically decreasing such that
    \begin{equation}
        \| T(x^k;d) - x^k \| \leq \mu_k, \ \ \ \mbox{for all $k\in\bbN$},
        \label{eq: ass-upper-bd-safeguard}
    \end{equation}
    and there exists $\zeta \in (0,1)$ such that
    \begin{equation}
        \mu_{k+1} \leq \zeta \mu_k, \ \ \ \mbox{whenever $x^{k+1}$ is an L2O update.}
        \label{eq: ass-geometric-safeguard}
    \end{equation}
\end{assumption}

Our proposed methods  for choosing the sequence $\{\mu_k\}$ satisfy Assumption \ref{ass: mu-k} (see Table \ref{table: safeguarding-choices}).
These methods are {\it adaptive} in the sense that each     $\mu_k$ depends upon the weights $\Theta^k$,  iterate $x^k$ and (possibly) previous weights and iterates.
Each safeguard parameter $\mu_k$ also remains constant in $k$ except for when the sum of residual norms $\|x^{k+1}-T(x^{k+1})\|$ and $\|y^k-x^k\|$ decreases to less than a geometric factor of $\mu_k$.
This allows each $\mu_k$ to trail the value of the residual norm $\|x^k-T(x^k)\|$ and   the residual norm to increase in $k$ from time to time.  This trailing behavior provides flexibility to the L2O updates. 
Our main result is below and   is followed by a corollary justifying use of the schemes in Table  \ref{table: safeguarding-choices} (both proven in  the appendix). 

\def\thmConv{
	If $\{x^k\}$ is a sequence generated by the inner loop in Safe-L2O and Assumptions \ref{ass: T} and \ref{ass: mu-k} hold,	then $\{x^k\}$ converges to a limit   $x_d^\star \in \fix{T(\cdot;d)}$, \ie $x^k \rightarrow x_d^\star$. 
}
\begin{theorem} \label{thm: Safe-L2O-convergence}
    \thmConv
\end{theorem}

\begin{figure*}
    \centering
    \subfloat[Performance on {\it seen} distribution]{\includegraphics[height=2.in]{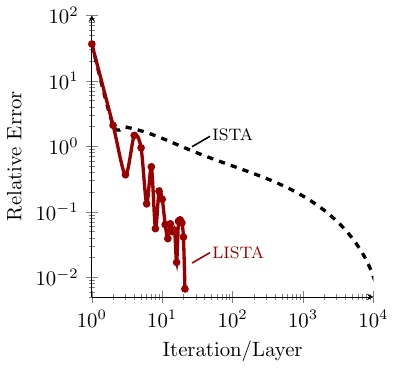}}
    \subfloat[Performance on {\it unseen} distribution]{\includegraphics[height=2.15in]{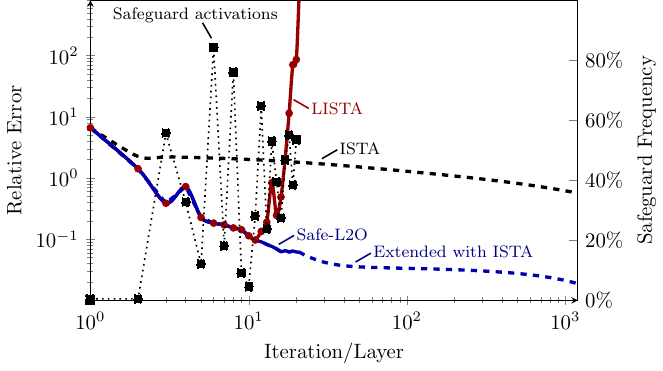}}
    \caption{Plot of error versus iteration for LISTA denoising.
     {Here ISTA is the classic algorithm, LISTA is the L2O operator in Algorithm~\ref{alg: L2O_Abstract} and Safe-L2O is the safeguarded version of LISTA in Algorithm~\ref{alg: Safe-L2O}.}
	Trained with $\phi_d = f_d$. Inferences used $\alpha = 0.99$ and $\mbox{EMA}(0.25)$.  In (b), how often the L2O update is ``bad'' and the  safeguard activates for Safe-L2O is indicated in reference to the right vertical axis. This plot shows the safeguard is used intermittently for $k > 2$.} 
    \label{fig: LISTA-plots}
\end{figure*}

\def\corConv{
If   $\{x^k\}$ is   generated by the inner loop in Safe-L2O and Assumption \ref{ass: T} holds, and  $\{\mu_k\}$ is generated using a scheme outlined in Table \ref{table: safeguarding-choices} with $\alpha \in [0,1)$ and $\beta\in(0,\infty)$, then  $x^k \rightarrow x_d^\star \in \fix{T(\cdot;d)}$.
}

\begin{corollary} \label{corollary: safeguarding-results}
	\corConv
\end{corollary}

\begin{table*}[t]
 	\caption{Rules for updating $\mu_k$. Here $\alpha,\theta\in(0,1)$, $\beta\in(0,\infty)$, and 
 	$C_k$ is the statement   $\|y^k-T(y^{k};\ d)\| +\beta \|y^k - x^k\|\leq \alpha\mu_k$. 
 	}
 	\label{table: update-mu}
 	\vspace{2pt}
 	\label{table: safeguarding-choices} 
 	\centering 
 	\begin{tabular}{c|c} 
 		{\sc Name} & {\sc Update Formula}
 		\\[3 pt]\hline 		
 		\begin{minipage}{0.2\textwidth}
 	        \begin{center}
 	            Geometric Sequence \\
 	        GS($\theta$)
 	        \end{center}
 	    \end{minipage}
 	    & 
 	    
 	    \begin{minipage}{0.65\textwidth}
 	    \vspace*{3pt}
 	    \begin{center}
 		$\mu_{k+1} = \begin{cases}\begin{array}{cl}
		 \theta \mu_k &  \mbox{if $C_k$ holds,} \\ 
 		\mu_k &   \mbox{otherwise.}
 		\end{array}   				
 		\end{cases}
 		$ \\[2pt]
 		{\it Decrease $\mu_k$ by   factor $\theta$ for ``good'' residuals.}
 		\vspace*{4 pt}
 		\end{center} 		
 		\end{minipage}
 		\\\hline
 		  		
 		\begin{minipage}{0.2\textwidth}
 	        \begin{center}
 	            Recent Term \\
 	            RT
 	        \end{center}
 	    \end{minipage} 		 
 	    &
 	    \begin{minipage}{0.70\textwidth}
 	    \vspace*{3pt}
 	    \begin{center}
 		$\mu_{k+1} = \begin{cases}\begin{array}{cl}
		 	\|x^{k+1}-T(x^{k+1};d)\| + \beta \|x^{k+1}-x^k\| &  \mbox{if $C_k$ holds,} \\ 
 		\mu_k &   \mbox{otherwise.}
 		\end{array}   				
 		\end{cases}
 		$ \\[2pt]
 		{\it Take $\mu_k$ to be most recent ``good'' residual.}
 		\vspace*{4pt}
 		\end{center}
 		\end{minipage}
 		\\\hline
 		
 		
 		\begin{minipage}{0.20\textwidth}
 	        \begin{center}
 	            Exponential\\  
 	            Moving Average \\ 	            
 	            EMA($\theta$)
 	        \end{center}
 	    \end{minipage} 		 
 	    &
 	    \begin{minipage}{0.73\textwidth}
 	    \vspace*{3pt}
 	    \begin{center}
 		$\mu_{k+1} = \begin{cases}\begin{array}{cl}
		 	\theta \left( \|x^{k+1}-T(x^{k+1};d)\| +   \beta \|x^{k+1}-x^k\|\right)+ (1-\theta)\mu_{k-1}  &  \mbox{if $C_{k}$ holds,} \\ 
 		\mu_k &   \mbox{otherwise.}
 		\end{array}   				
 		\end{cases}
 		$ \\[2pt]
 		{\it  Exponentially average $\mu_k$ with the latest ``good'' residuals.}
 		\vspace*{4pt}
 		\end{center}
 		\end{minipage}
 		\\\hline

 	    
 	\end{tabular} 
 \end{table*}

We summarize the safeguard schemes in Table \ref{table: safeguarding-choices} as follows. The GS method decreases $\mu_k$ be a fixed geometric factor at each update. The EMA method exponentially  averages all past and current residual sums where $\mu_k$ is/was modified.  The RT method sets $\mu_k$ to be the  last     residual norm sum to be ``good'', \ie satisfy the $C_k$ inequality. We find   EMA to be the most practical safeguard due to its adaptive nature.
    
\begin{remark}
 {The appropriate frequency for the safeguard to trigger can be estimated by tuning L2O parameters for optimal performance on a training set {\it without} safeguarding and then using a validation set to test various safeguards with the L2O scheme. To avoid possible confusions, note \textit{we are not trying to prove the convergence of any standalone L2O algorithm}. We instead 1) alarm on an L2O update when it may break convergence, 2) replace it with a fallback update, and 3) show the resulting ``hybrid optimization" converges to a solution of the provided optimization problem.}
\end{remark}


\section{Training and Averaged Operator Selection} 
\label{sec: NN}

Safe-L2O may be executed via inferences of a feed forward neural network. 
The input into the network is the data $d$, often in vector form.
Input $d$ is usually the \textit{observation} we have, based on which we optimize over the variable of interest. 
For example,   the LASSO problem   in (\ref{eq: LASSO-problem}) is used for  sparse coding, where the goal is to recover a unknown sparse vector $x^\star$ from its noisy measurements $d = Ax^\star + \varepsilon$. Assuming $A$ is known beforehand, the input to the Safe-L2O model is the observation $d$. In other cases, the dictionary $A$ can change and also be part of the input to the model. We include case-by-case discussions about what the inputs   for each numerical example.

Each layer of the Safe-L2O model is designed so that its input is $x^k$, to which it applies either an L2O or fallback update (following the Safe-L2O method), and outputs $x^{k+1}$ to the next layer. 
The set  over which  $\Theta$ is minimized, may be chosen with great flexibility.
For each application of the algorithm, the fallback operator  depends upon the  data $d$. 

The ``optimal'' choice of parameters $\Theta$ depends upon the application.  
Suppose each $d$ is drawn from a common distribution $\mathcal{D}$.   
Then a choice of ``optimal'' parameters $\Theta^\star$   may be identified as those for which the expected value of $\phi(x^K;d)$ is minimized among $d\sim \mathcal{D}$,
where $\phi(\cdot\ ;d):\bbR^n\rightarrow \bbR$ is an appropriate cost function and $K$ is a fixed positive integer.
Mathematically, this means   $\Theta^\star$  solves the problem
\begin{equation}
\min_{\Theta} \mathbb{E}_{d\sim \mathcal{D}}[  \phi(x^K(\Theta;\ d);\ d) ],
\label{eq: learning-problem}
\end{equation}
where we emphasize the dependence of $x^K$ on $\Theta$ and $d$   by writing $x^K = x^K(\Theta;\ d)$. 
Examples for $\phi$ include the original objective function (\ie  $\phi(x;\ d)=f(x;\  d)$) and the fixed point residual $\|x-T(x;\ d)\|$. 
We approximately solve (\ref{eq: learning-problem})  by sampling data $\{d^n\}_{n=1}^N$   from   $\mathcal{D}$ and minimizing an empirical loss function.
Summaries for   training   are outlined in the appendices.
Note different learning problems than (\ref{eq: learning-problem}) may be used (\eg  the min-max problem used by adversarial networks \cite{goodfellow2014generative}). 

 \begin{figure*}
    \centering
    \subfloat[Performance on {\it seen} distribution]{\includegraphics[height=2.25in]{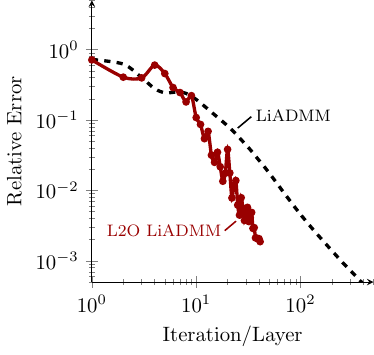}}
    \subfloat[Performance on {\it unseen} distribution]{\includegraphics[height=2.35in]{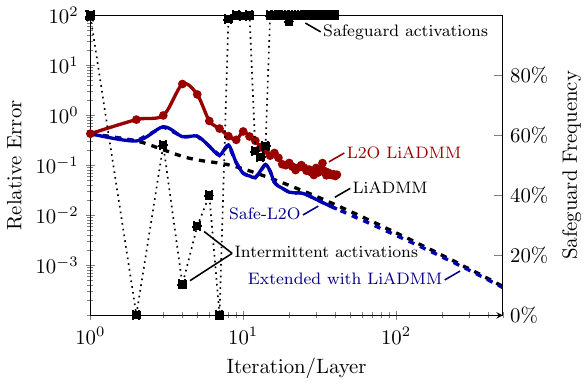}}
    \caption{Plot of error versus iteration for LiADMM.
     {Here LiADMM is the classic algorithm, L2O LiADMM is the L2O operator in Algorithm~\ref{alg: L2O_Abstract} and Safe-L2O is the safeguarded version of L2O LiADMM in Algorithm~\ref{alg: Safe-L2O}.}
  Inferences used $\alpha = 0.99$ and $\mbox{EMA}(0.75)$. In (b), how often the L2O update is ``bad'' and the  safeguard activates for Safe-L2O is indicated in reference to the right vertical axis. The safeguard is used about 10\% and 30\% of the time when $k=4$ and $k=5$, respectively.}
    \label{fig: LiADMM-plots}
\end{figure*}

\section{Numerical Examples}
\label{sec: numerical-examples} 

This section presents examples using Safe-L2O.\footnote{All of the codes for this work can be found on GitHub here: (link will be added after review process).} We numerically investigate 
i) the convergence rate of Safe-L2O relative to corresponding conventional algorithms,
ii) the efficacy of safeguarding procedures when inferences are performed on data for which L2O fails intermittently, 
and iii) the convergence of Safe-L2O schemes even when the application of  {L2O operators} is not theoretically justified. 
We first use $\LTO$ from ALISTA \cite{LiuChenWangYin2019_alista} on a synthetic LASSO problem. We then use LISTA on a LASSO problem for image processing, differentiable linearized ADMM   \cite{XieWuZhongLiuLin2019_differentiable} on a sparse coding problem.
In all   three types of problems, the input data $d$ to the L2O models is
 {generated by $d = Ax^\star + \varepsilon$ where $\varepsilon$ is white Gaussian noise and $x^\star$ is the hidden variable that we want to recover through a dictionary $A$, which is generated and beforehand and fixed.}
In these experiments, the ground-truth vector $x^\star$ are sampled from a distribution, which characterizes the distribution of data of interest along with the dictionary. We denote the distribution of the input data as $\mathcal{D}$. Besides these ``linear'' examples, we also validate Safe-L2O on a distribution of LASSO problems where the dictionary $A$ also changes and is part of the input to the L2O models. In this case, the distributions of the dictionaries and ground-truth vectors together characterize the input distribution $\mathcal{D}$.

In each example, $f_d^\star$ denotes the optimal value of the objective $f(x;d)$ among all possible $x$. 
Performance  is measured using a modified relative objective error:
\begin{equation}
    \mbox{Relative Error} = \mathcal{R}_{f,\mathcal{D}}(x) \triangleq   
    \dfrac{ \mathbb{E}_{d\sim \mathcal{D}} [f(x;\ d) - f_d^\star] }{\mathbb{E}_{d\sim \mathcal{D}}[ f_d^\star] },
    \label{eq: relative-error-formula}
\end{equation} 
where the expectations are estimated numerically (see the appendices for details).  
We use (\ref{eq: relative-error-formula}) rather than the expectation of relative error to avoid high sensitivity to outliers.  

Our numerical results are shown in several plots. 
When each iterate $x^k$ is computed using data $d$ drawn from the same distribution $\mathcal{D}_s$ that was used to train the L2O algorithm, we say the performance is on the ``seen'' distribution $\mathcal{D}_s$.  
These plots form the primary illustrations of the speedup of L2O algorithms as compared to conventional optimization algorithms.
When each $d$ is drawn from a distribution ${\mathcal{D}}_u$ that is {\it different} than $\mathcal{D}_s$, we  refer to ${\mathcal{D}}_u$ as the {\it unseen} distribution.
These plots show the ability of the safeguard to ensure convergence. A dotted plot with square markers is also added to show the  frequency of safeguard activations among test samples, with the reference axis on the right hand side of the plots.
We extend the Safe-L2O methods beyond their training iterations by applying the fallback operator $T$; we demarcate where this extension begins by changing the Safe-L2O plots from solid to dashed lines. 
As Safe-L2O convergence holds whenever $\beta > 0$, we can set $\beta$ to be arbitrarily small (\eg below machine precision); for simplicity, we use $\beta=0$ in the experiments (as, even in this case, it can be shown that iterates approach the solution set).
Implementation details for each experiment are   in the appendices.

\subsection{ALISTA for LASSO}\label{subsec: ALISTA}
Here we consider the  LASSO problem for sparse coding.
Let $x^\star \in \bbR^{500}$ be a sparse vector and   $A\in \bbR^{250\times 500}$ be a dictionary.
We assume access is given to noisy linear measurements $d\in \bbR^{250}$, where $\varepsilon\in\bbR^{250}$ is additive Gaussian white noise and
$
d = Ax^\star + \varepsilon.
$
Even for underdetermined systems, when $x^\star$ is sufficiently sparse and  $\tau \in (0,\infty)$ is an appropriately chosen regularization parameter, $x^\star$ can often be reasonably estimated by solving the LASSO problem 
\begin{equation}
\min_{x\in\bbR^n} f(x;\ d) \triangleq \dfrac{1}{2}\|Ax-d\|_2^2 + \tau \|x\|_1,
\label{eq: LASSO-problem}
\end{equation}
where $\|\cdot\|_2$ and $\|\cdot\|_1$ are the $\ell_2$ and $\ell_1$ norms, respectively.
A classic method for solving (\ref{eq: LASSO-problem}) is the iterative shrinkage thresholding algorithm (ISTA) (\eg see \cite{DaubechiesDefriseMol2004_iterative}).\footnote{This is a special case of the proximal-gradient in Table \ref{table: NE-operators}.} \citet{LiuChenWangYin2019_alista} present the L2O scheme ALISTA that we implement here. 
This L2O model $\LTO$ is parameterized by $\Theta^k = (\theta_k,\gamma_k) \in \bbR^2$. \subsection{Linearized ADMM} \label{subsec: LiADMM}
 Let $A\in\bbR^{250\times 500}$ and $d\in\bbR^{250}$ be as in the LASSO problem.
 Here we  apply the L2O scheme differentiable linearized ADMM (LiADMM) of \citet{XieWuZhongLiuLin2019_differentiable} to the closely related sparse coding problem
 \begin{equation}
 \min_{x\in\bbR^n} \|Ax-d\|_1 + \tau \|x\|_1.
 \label{eq: LiADMM-problem}
 \end{equation}
 The L2O  network $\LTO$ and fallback linearized ADMM (LiADMM) operator $T$ are provided in the appendices along with implementation details.
 Plots are provided in Figure \ref{fig: LiADMM-plots}.
 
\begin{figure}
    \centering
    \subfloat[Performance on {\it seen} distribution]{\includegraphics[height=2.in]{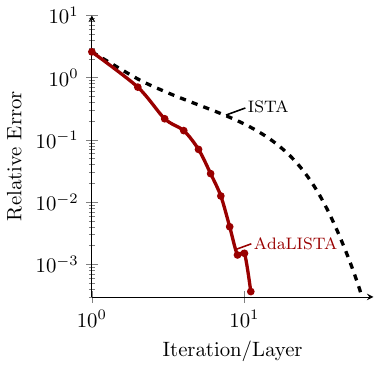}}
    \caption{Plot of error versus iteration for AdaLISTA example.
     {Here ISTA is the classic algorithm, AdaLISTA is the L2O operator in Algorithm~\ref{alg: L2O_Abstract} and Safe-L2O is the safeguarded version of AdaLISTA in Algorithm~\ref{alg: Safe-L2O}.}
}
    \label{fig:AdaLISTA-plots-seen}
\end{figure}

 \subsection{LISTA for Natural Image Denoising} \label{subsec: RW-LISTA}
To evaluate our safeguarding mechanism in a more realistic setting, we apply safeguarded LISTA to a natural image denoising problem. In this subsection, we learn a LISTA-CP model \cite{ChenLiuWangYin2018_theoretical} to perform natural image denoising.
During training,   L2O LISTA-CP model is trained to recover clean images from their Gaussian noisy counterparts by solving (\ref{eq: LASSO-problem}).
In (\ref{eq: LASSO-problem}), $d$ is the noisy input to the model, and the clean image is recovered with $\hat{d} = Ax^\star$, where $x^\star$ is the optimal solution. The dictionary $A\in\bbR^{256\times 512}$ is learned on the BSD500 dataset \cite{BSD500} by solving a dictionary learning problem \cite{xu2014fast}.
During testing, however, the learned L2O LISTA-CP is applied to unseen pepper-and-salt noisy images. Comparison plots are provided in Figure \ref{fig: LISTA-plots}.

 \begin{figure}
    \centering
    \subfloat[Performance on {\it unseen} distribution]{\includegraphics[width=0.47\textwidth]{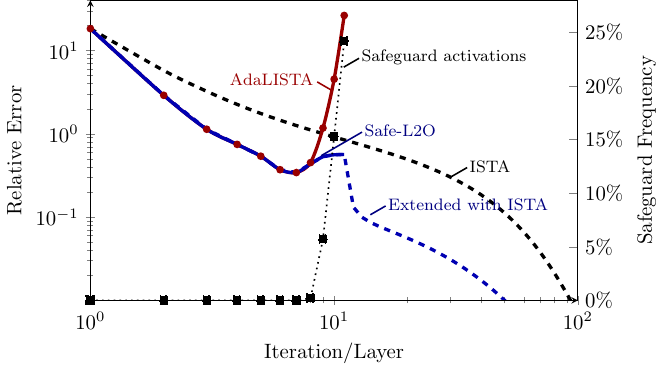}}
    \caption{Plot of error versus iteration for AdaLISTA example.
     {Here ISTA is the classic algorithm, AdaLISTA is the L2O operator in Algorithm~\ref{alg: L2O_Abstract} and Safe-L2O is the safeguarded version of AdaLISTA in Algorithm~\ref{alg: Safe-L2O}.}
	Inferences used $\alpha = 0.99$ and $\mbox{EMA}(0.25)$.  In (b), how often the L2O update is ``bad'' and the  safeguard activates for Safe-L2O is indicated in reference to the right vertical axis. This plot shows the safeguard is used intermittently for $k > 7$.
}
    \label{fig:AdaLISTA-plots-unseen}
\end{figure}

\subsection{AdaLISTA: Dictionary as Part of Inputs}
\label{subsec: AdaLISTA}
Here we consider the same LASSO problem (\ref{eq: LASSO-problem}) as in Subsection \ref{subsec: ALISTA} but make the dictionary $A$ part of the inputs to the L2O model (\ie able to change across samples).
 \citet{aberdam2021ada} present a new L2O scheme AdaLISTA that is trained to quickly solve a distribution of LASSO problems with varying dictionaries. AdaLISTA has a different parameterization scheme from the original LISTA \cite{GregorLeCun2010_learning} to enable the adaptivity to the dictionaries. The L2O model $\LTO$ in AdaLISTA is parameterized by $\zeta = (\theta,\gamma) \in \bbR^2$ and two weight matrices $W_1, W_2\in\bbR^{m\times m}$ where the dictionary $A$ has shape   $m\times n$. The two matrices are shared by operators in all iterations.

We mainly follow the settings in \cite{aberdam2021ada}. Specifically, we let $x^\star\in\bbR^{70}$ be sparse vectors with random supports of cardinality $s=6$ and a single \textbf{fixed} dictionary $A^\prime\in\bbR^{50\times70}$. We assume access is given to noiseless linear measurements $d\in\bbR^{50}=Ax^\star$, where $A$ is uniformly sampled from all column-permuted variants of $A^\prime$.   Figure~\ref{fig:AdaLISTA-plots-seen} and \ref{fig:AdaLISTA-plots-unseen} show summary plots, from which we have similar observations as in the ALISTA experiments with a single fixed dictionary. 

\section{Conclusions}
\label{sec: conclusion}

Numerous insights may be drawn from our examples. The first observation is, roughly speaking,  each L2O scheme in our numerical examples  reduces computational costs by at least one order of magnitude when applied to data from the same distribution as the training data (as compared to analytic optimization algorithms). This is consistent with results of previous works. More importantly, plots in   Figures \ref{fig: ALISTA-plots}b and \ref{fig: LISTA-plots}b, and \ref{fig:AdaLISTA-plots-unseen} show   the safeguard steers updates to convergence when they would otherwise diverge or converge slower than the conventional algorithm. That is,  Safe-L2O converges with data distinct from training while the nonsafeguarded L2O schemes diverge.
 
This work proposes a framework for ensuring convergence of L2O algorithms.
Sequences generated by our Safe-L2O method provably converge to solutions of the optimization problems.   Our  Safe-L2O algorithm is also easy to implement as a wrapper around trained neural networks. Numerical experiments demonstrate rapid convergence by Safe-L2O methods and effective safeguarding when the L2O schemes appear to otherwise diverge. 
Future work will provide a better data-driven fallback method  and investigate stochastic extensions.

\newpage
\bibliography{Safe-L2O-AAAI-23}

\begin{thebibliography}{30}
\providecommand{\natexlab}[1]{#1}

\bibitem[{Aberdam, Golts, and Elad(2021)}]{aberdam2021ada}
Aberdam, A.; Golts, A.; and Elad, M. 2021.
\newblock Ada-lista: Learned solvers adaptive to varying models.
\newblock \emph{IEEE Transactions on Pattern Analysis and Machine
  Intelligence}.

\bibitem[{Ablin et~al.(2019)Ablin, Moreau, Massias, and
  Gramfort}]{AblinMoreauMassiasGramfort2019_learning}
Ablin, P.; Moreau, T.; Massias, M.; and Gramfort, A. 2019.
\newblock Learning step sizes for unfolded sparse coding.
\newblock \emph{arXiv:1905.11071}.

\bibitem[{Bauschke and Combettes(2017)}]{BauschkeCombettes2017_convex}
Bauschke, H.; and Combettes, P. 2017.
\newblock \emph{Convex {Analysis} and {Monotone} {Operator} {Theory} in
  {Hilbert} {Spaace}}.
\newblock Springer, 2nd. edition.

\bibitem[{Beck(2017)}]{Beck2017_First}
Beck, A. 2017.
\newblock \emph{First-order methods in optimization}, volume~25.
\newblock SIAM.

\bibitem[{Blumensath and Davies(2009)}]{BlumensathDavies2009_iterative}
Blumensath, T.; and Davies, M.~E. 2009.
\newblock Iterative hard thresholding for compressed sensing.
\newblock \emph{Applied and Computational Harmonic Analysis}, 27(3): 265--274.

\bibitem[{Borgerding, Schniter, and
  Rangan(2017)}]{BorgerdingSchniterRangan2017_ampinspired}
Borgerding, M.; Schniter, P.; and Rangan, S. 2017.
\newblock {AMP}-{Inspired} {Deep} {Networks} for {Sparse} {Linear} {Inverse}
  {Problems}.
\newblock \emph{IEEE Transactions on Signal Processing}, 65(16): 4293--4308.

\bibitem[{Chen et~al.(2022)Chen, Chen, Chen, Heaton, Liu, Wang, and
  Yin}]{chen2022learning}
Chen, T.; Chen, X.; Chen, W.; Heaton, H.; Liu, J.; Wang, Z.; and Yin, W. 2022.
\newblock Learning to Optimize: A Primer and A Benchmark.
\newblock \emph{Journal of Machine Learning Research}, 23(189): 1--59.

\bibitem[{Chen et~al.(2018)Chen, Liu, Wang, and
  Yin}]{ChenLiuWangYin2018_theoretical}
Chen, X.; Liu, J.; Wang, Z.; and Yin, W. 2018.
\newblock Theoretical {Linear} {Convergence} of {Unfolded} {ISTA} and {Its}
  {Practical} {Weights} and {Thresholds}.
\newblock In Bengio, S.; Wallach, H.; Larochelle, H.; Grauman, K.;
  Cesa-Bianchi, N.; and Garnett, R., eds., \emph{Advances in {Neural}
  {Information} {Processing} {Systems} 31}, 9061--9071. Curran Associates, Inc.

\bibitem[{Daubechies, Defrise, and
  Mol(2004)}]{DaubechiesDefriseMol2004_iterative}
Daubechies, I.; Defrise, M.; and Mol, C.~D. 2004.
\newblock An iterative thresholding algorithm for linear inverse problems with
  a sparsity constraint.
\newblock \emph{Communications on Pure and Applied Mathematics}, 57(11):
  1413--1457.

\bibitem[{Giryes et~al.(2018)Giryes, Eldar, Bronstein, and
  Sapiro}]{GiryesEldarBronsteinSapiro2018_tradeoffs}
Giryes, R.; Eldar, Y.~C.; Bronstein, A.~M.; and Sapiro, G. 2018.
\newblock Tradeoffs {Between} {Convergence} {Speed} and {Reconstruction}
  {Accuracy} in {Inverse} {Problems}.
\newblock \emph{IEEE Transactions on Signal Processing}, 66(7): 1676--1690.

\bibitem[{Goodfellow et~al.(2014)Goodfellow, Pouget-Abadie, Mirza, Xu,
  Warde-Farley, Ozair, Courville, and Bengio}]{goodfellow2014generative}
Goodfellow, I.; Pouget-Abadie, J.; Mirza, M.; Xu, B.; Warde-Farley, D.; Ozair,
  S.; Courville, A.; and Bengio, Y. 2014.
\newblock Generative adversarial nets.
\newblock In \emph{Advances in neural information processing systems},
  2672--2680.

\bibitem[{Gregor and LeCun(2010)}]{GregorLeCun2010_learning}
Gregor, K.; and LeCun, Y. 2010.
\newblock Learning {Fast} {Approximations} of {Sparse} {Coding}.
\newblock In \emph{Proceedings of the 27th {International} {Conference} on
  {International} {Conference} on {Machine} {Learning}}, {ICML}'10, 399--406.
  USA: Omnipress.

\bibitem[{Hershey, Roux, and Weninger(2014)}]{HersheyRouxWeninger2014_deep}
Hershey, J.~R.; Roux, J.~L.; and Weninger, F. 2014.
\newblock Deep {Unfolding}: {Model}-{Based} {Inspiration} of {Novel} {Deep}
  {Architectures}.
\newblock \emph{arXiv:1409.2574}.

\bibitem[{Krasnosel'ski\u{\i}(1955)}]{Krasnoselskii1955_two}
Krasnosel'ski\u{\i}, M. 1955.
\newblock Two remarks about the method of successive approximations.
\newblock \emph{Uspekhi Mat. Nauk}, 10: 123--127.

\bibitem[{Liu et~al.(2019)Liu, Chen, Wang, and Yin}]{LiuChenWangYin2019_alista}
Liu, J.; Chen, X.; Wang, Z.; and Yin, W. 2019.
\newblock {ALISTA}: Analytic Weights Are As Good As Learned Weights in {LISTA}.
\newblock In \emph{International Conference on Learning Representations}.

\bibitem[{Liu, Shen, and Gu(2019)}]{LiuShenGu2019_linearized}
Liu, Q.; Shen, X.; and Gu, Y. 2019.
\newblock Linearized {ADMM} for {Nonconvex} {Nonsmooth} {Optimization} {With}
  {Convergence} {Analysis}.
\newblock \emph{IEEE Access}, 7: 76131--76144.

\bibitem[{Mann(1953)}]{Mann1953_mean}
Mann, R. 1953.
\newblock Mean {Value} {Methods} in {Iteration}.
\newblock 4(3): 506--510.

\bibitem[{Martin et~al.(2001)Martin, Fowlkes, Tal, and Malik}]{BSD500}
Martin, D.; Fowlkes, C.; Tal, D.; and Malik, J. 2001.
\newblock A Database of Human Segmented Natural Images and its Application to
  Evaluating Segmentation Algorithms and Measuring Ecological Statistics.
\newblock In \emph{Proc. 8th Int'l Conf. Computer Vision}, volume~2, 416--423.

\bibitem[{Metzler, Mousavi, and
  Baraniuk(2017)}]{MetzlerMousaviBaraniuk2017_learned}
Metzler, C.; Mousavi, A.; and Baraniuk, R. 2017.
\newblock Learned {D}-{AMP}: {Principled} {Neural} {Network} based
  {Compressive} {Image} {Recovery}.
\newblock In Guyon, I.; Luxburg, U.~V.; Bengio, S.; Wallach, H.; Fergus, R.;
  Vishwanathan, S.; and Garnett, R., eds., \emph{Advances in {Neural}
  {Information} {Processing} {Systems} 30}, 1772--1783. Curran Associates, Inc.

\bibitem[{Moreau and Bruna(2017)}]{MoreauBruna2017_understanding}
Moreau, T.; and Bruna, J. 2017.
\newblock Understanding {Trainable} {Sparse} {Coding} with {Matrix}
  {Factorization}.

\bibitem[{Ryu and Yin(2022)}]{ryu2022large}
Ryu, E.; and Yin, W. 2022.
\newblock \emph{Large-Scale Convex Optimization: Algorithm Designs via Monotone
  Operators}.
\newblock Cambridge, England: Cambridge University Press.

\bibitem[{Sprechmann, Bronstein, and
  Sapiro(2015)}]{SprechmannBronsteinSapiro2015_learning}
Sprechmann, P.; Bronstein, A.~M.; and Sapiro, G. 2015.
\newblock Learning {Efficient} {Sparse} and {Low} {Rank} {Models}.
\newblock \emph{IEEE Transactions on Pattern Analysis and Machine
  Intelligence}, 37(9): 1821--1833.

\bibitem[{Themelis and Patrinos(2019)}]{ThemelisPatrinos2019_supermann}
Themelis, A.; and Patrinos, P. 2019.
\newblock {SuperMann}: a superlinearly convergent algorithm for finding fixed
  points of nonexpansive operators.

\bibitem[{Wang, Ling, and Huang(2016)}]{WangLingHuang2016_learning}
Wang, Z.; Ling, Q.; and Huang, T.~S. 2016.
\newblock Learning {Deep} $\ell_0$ {Encoders}.
\newblock In \emph{Thirtieth {AAAI} {Conference} on {Artificial}
  {Intelligence}}.

\bibitem[{Wang et~al.(2016)Wang, Liu, Chang, Ling, Yang, and
  Huang}]{WangLiuChangLingYangHuang2016_d3}
Wang, Z.; Liu, D.; Chang, S.; Ling, Q.; Yang, Y.; and Huang, T.~S. 2016.
\newblock D3: {Deep} {Dual}-{Domain} {Based} {Fast} {Restoration} of
  {JPEG}-{Compressed} {Images}.
\newblock 2764--2772.

\bibitem[{Xie et~al.(2019)Xie, Wu, Zhong, Liu, and
  Lin}]{XieWuZhongLiuLin2019_differentiable}
Xie, X.; Wu, J.; Zhong, Z.; Liu, G.; and Lin, Z. 2019.
\newblock Differentiable {Linearized} {ADMM}.
\newblock \emph{arXiv:1905.06179}.

\bibitem[{Xin et~al.(2016)Xin, Wang, Gao, Wipf, and
  Wang}]{XinWangGaoWipfWang2016_maximal}
Xin, B.; Wang, Y.; Gao, W.; Wipf, D.; and Wang, B. 2016.
\newblock Maximal {Sparsity} with {Deep} {Networks}?
\newblock In Lee, D.~D.; Sugiyama, M.; Luxburg, U.~V.; Guyon, I.; and Garnett,
  R., eds., \emph{Advances in {Neural} {Information} {Processing} {Systems}
  29}, 4340--4348. Curran Associates, Inc.

\bibitem[{Xu and Yin(2014)}]{xu2014fast}
Xu, Y.; and Yin, W. 2014.
\newblock A fast patch-dictionary method for whole image recovery.
\newblock \emph{arXiv preprint arXiv:1408.3740}.

\bibitem[{Yang et~al.(2016)Yang, Sun, Li, and Xu}]{YangSunLiXu2016_deep}
Yang, Y.; Sun, J.; Li, H.; and Xu, Z. 2016.
\newblock Deep {ADMM}-{Net} for {Compressive} {Sensing} {MRI}.
\newblock In Lee, D.~D.; Sugiyama, M.; Luxburg, U.~V.; Guyon, I.; and Garnett,
  R., eds., \emph{Advances in {Neural} {Information} {Processing} {Systems}
  29}, 10--18. Curran Associates, Inc.

\bibitem[{Zhang, O'Donoghue, and Boyd(2018)}]{ZhangODonoghueBoyd2018_Globally}
Zhang, J.; O'Donoghue, B.; and Boyd, S. 2018.
\newblock Globally convergent type-I Anderson acceleration for non-smooth
  fixed-point iterations.
\newblock \emph{arXiv:1808.03971}.

\end{thebibliography}


\newpage
\appendix
\section{Numerical Example Supplement Materials} 

We begin with a general note on training and then consider individual experiments.
The training procedure for each experiment was conducted layerwise. By this, we mean that first the network weights were tuned using one layer. Then the network was trained with two layers, using the learned weights from the first layer as an initialization/warm start for that layer's parameters. This was then repeated until the final number $K$ of layers was reached.
This approach is built upon the intuition that, because the network layers model an optimization algorithm that progressively improves, each successive layer's weights likely depend upon the previous layers weights. \\

\noindent {\bf Supplement for ALISTA.}
In similar manner to \cite{ChenLiuWangYin2018_theoretical} and \cite{LiuChenWangYin2019_alista}, we use the following setup.
We take $m = 250$, $n=500$, and {$\tau = 0.001$}.
Each entry of the dictionary $A$ is sampled i.i.d from the standard Gaussian distribution, i.e., $a_{ij} \sim \mathcal{N}(0,1/m)$.
Having these entries, we then normalize each column of $A$, with respect to the Euclidean norm.
Each $d$ in the distribution $\mathcal{D}_s$ of data used to train the neural network is constructed using $d=Ax^\star+\varepsilon$ with noise $\varepsilon \sim 0.1\cdot \mathcal{N}(0,1/m)$ and each entry of $x^\star$ as the composition of Bernoulli and Gaussian distributions, i.e., $x^\star_j \sim \mbox{Ber}(0.1) \circ \mathcal{N}(0,1)$ for all $j \in [n]$.  
Each $d$ in the {\it unseen} distribution $\mathcal{D}_u$ is computed using the same distribution of noise $\varepsilon$ as before and using  $x_j^\star \sim \mbox{Ber}(0.2)\circ \mathcal{N}({0},2)$.
Our data set consists of 10,000 training samples and 1,000 test samples.

Given $x^1\in \bbR^n$, the ISTA method iteratively computes 
\begin{equation} 
x^{k+1} 
\triangleq  T(x^k)
\triangleq   \eta_{\tau/L} \left( x^k -\dfrac{1}{L} A^T(Ax^k - d) \right),
\label{eq: T-ISTA}
\end{equation}
where $L=\|A^t A\|_2$ and $\eta_\theta$ is the soft-thresholding function defined by component-wise operations:
\begin{equation}
\eta_\theta(x) \triangleq   \mbox{sgn}(x) \cdot \max\{ 0, |x|-\theta \}.
\label{eq: soft-thresholding-def-scalar}
\end{equation}

We applied  Safe-L2O to the LASSO problem above by using $T$ defined in (\ref{eq: T-ISTA}) and the $\LTO$ update operation from ALISTA \cite{LiuChenWangYin2019_alista}.
The L2O operator $\LTO$ is parameterized by $\zeta = (\theta,\gamma)$ for positive scalars $\theta$ and $\gamma$ and defined by 
\begin{equation} 
\LTO(x;\ \zeta)
\triangleq   \eta_{\theta}\left( x - \gamma W^T (Ax - d) \right),
\label{eq: ALISTA}
\end{equation}
where 
\begin{align}
& W \in  \argmin_{M \in \bbR^{m\times n}} \|M^T A\|_F, \nonumber  \\
& \mbox{subject to}~ (M_{:,\ell})^T A_{:,\ell} = 1, \ \ \mbox{for all $\ell\in [n]$,}
\label{eq: LASSO-W-def}
\end{align} 
and $\|\cdot\|_F$ is the Frobenius norm and the Matlab notation $M_{:,\ell}$ is used to denote the $\ell$th column of the matrix $M$. 
The parameter $\Theta = (\theta^k,\gamma^k)_{k=1}^K$ consists of $2K$ scalars.   \\

\noindent {\bf Supplement for LiADMM.} 
LiADMM is used to solve problems of the form
\begin{equation}
    \min_{x\in\bbR^n,z\in\bbR^m} f(x) +  g(z) 
    \ \ \ \mbox{s.t.} \ \ \ 
    Ax + Bz = d,
    \label{eq: ADMM-problem}
\end{equation}
for which LiADMM  generates sequences $\{x^k\}$, $\{z^k\}$ and $\{u^k\}$ defined by the updates
{\small 
    \begin{align}
        x^{k+1} 
        & \triangleq   \prox{\beta f}\left(x^k - \beta A^T \left[ u^k + \alpha \left(Ax^k + Bz^k - d\right) \right]\right), \nonumber\\ 
        z^{k+1}
        & \triangleq   \prox{\gamma g}\left( z^k - \gamma B^T \left[ u^k + \alpha \left(Ax^{k+1} + Bz^k - d\right) \right] \right), \nonumber\\
        u^{k+1} 
        & \triangleq   u^k + \alpha \left(Ax^{k+1} + Bz^{k+1} - d\right),\label{eq: LADMM-update}
    \end{align}
}with given scalars $\alpha,\beta,\gamma \in (0,\infty)$.
The problem (\ref{eq: LiADMM-problem}) may be written in the form of (\ref{eq: ADMM-problem}) by taking $f =   \tau \|\cdot\|_1$, $g = \|\cdot\|_1$, and $B = -\mbox{Id}$. 
In this case, the proximal operators in (\ref{eq: LADMM-update})  reduce to soft-thresholding operators. 
Although not given in Table \ref{table: NE-operators}, at each iteration of LiADMM there is an associated iterate $\nu^k$ for which the update $\nu^{k+1}$ is generated by applying an averaged operator $T$ to the current iterate $\nu^k$. As shown in Lemma \ref{lemma: LADMM-fixed-point-residual} of Subsection A.2, for our setup, assuming $ \gamma =1/\alpha$ and $\alpha \beta \|A^TA\|_2 < 1$, the norm of the associated fixed point residual at the iterate $\nu^k$ is given by 
    \begin{equation}
        \|\nu^k  - T(\nu^{k})\|
        =   \left\|  \left[ \begin{array}{c} Ax^{k+2} - z^{k+1} - d \\ P(x^{k+2}-x^{k+1})  \end{array} \right]\right\|,
         \label{eq: LADMM-S}
\end{equation}  
with $P$ defined below in (\ref{eq: PQ-matrix-defs}). 
For notational clarity, the term $x^k$ in the SKM and LSKM schemes is replaced in this subsection by the tuple $(x^k,z^k,u^k)$. 
This is of practical importance too since it is the sequence $\{x^k\}$ that converges to a solution of (\ref{eq: LiADMM-problem}).

We now modify the iteration (\ref{eq: LADMM-update}) for the problem (\ref{eq: LiADMM-problem}) to create the LiADMM L2O scheme.
We generalize soft-thresholding to  vectorized soft-thresholding for $\beta  \in \bbR^n$ by
\begin{equation}
    \eta_\beta (x)
    = (\eta_{\beta_1}(x_1),\ \eta_{\beta_2}(x_2),\ \ldots,\ \eta_{\beta_n}(x_n)).
    \label{eq: soft-thresholding-def-vector}
\end{equation}
We assume $\eta_\beta$ represents the scalar soft-thresholding in (\ref{eq: soft-thresholding-def-scalar}) when $\beta\in\bbR$ and the vector generalization (\ref{eq: soft-thresholding-def-vector}) when $\beta\in\bbR^n$.  	  
Combining ideas from ALISTA \cite{LiuChenWangYin2019_alista} and LiADMM \cite{LiuShenGu2019_linearized},  given $(x^k, z^k, \nu^k)\in\bbR^n\times\bbR^m\times\bbR^m$, $\alpha^k, \gamma^k, \xi^k \in \bbR^m$, $\beta^k, \sigma^k \in \bbR^n$, $W_1 \in \bbR^{n\times m}$, and $W_2 \in \bbR^{m\times m}$, set
{\footnotesize 
\begin{equation}
    \begin{aligned}
        \tilde{x}^{k+1} 
        & \triangleq   \eta_{\beta^k  }\left(x^k - \sigma^k \circ (W_1^k)^T \left[ \nu^k + \alpha_k \circ  \left(Ax^k - z^k - d\right) \right]\right), \\
        \tilde{z}^{k+1}
        & \triangleq   \eta_{ \gamma^k}\left( z^k - \xi^k \circ (W_2^k)^T \left[ \nu^k + \alpha_k\circ \left(A\tilde{x}^{k+1}-z^k - d\right) \right] \right), \\
        \tilde{\nu}^{k+1} 
        & \triangleq   \nu^k + \alpha_k  \circ \left(A\tilde{x}^{k+1} - \tilde{z}^{k+1} - d\right),
    \end{aligned}
    \label{eq: LADMM-update-learnable} 
\end{equation}}
with  element-wise products denoted by $\circ$. 
For the parameter $\zeta^k \triangleq   (  \alpha^k, \beta^k, \gamma^k, \sigma^k, \xi^k, W_1^k, W_2^k )$, then define 
\begin{equation}
    \LTO(x^k,z^k,\nu^k;\ \zeta^k)
    \triangleq   (\tilde{x}^{k+1}, \tilde{z}^{k+1}, \tilde{\nu}^{k+1}).
    \label{eq: LADMM-TL2O}
\end{equation}	
Fixing the number of iterations $K$, the learnable parameters from (\ref{eq: LADMM-update-learnable}) used in the LSKM Algorithm may be encoded by  $
\Theta = (\zeta^k)_{k=1}^K= \left(  \alpha^k, \beta^k, \gamma^k, \sigma^k, \xi^k, W_1^k, W_2^k\right)_{k=1}^K,$ 
consisting of $(2n + 3m+mn+m^2)K$ scalars. To stabilize the training process, we share the $W_1$ across all layers in practice. We also fix $W_2=-\mbox{Id}$ and only learn the step sizes $\xi^k$ before it. Moreover, different from other experiments, we add an additional end-to-end training stage after we finish the normal layer-wise training described at the beginning of the Appendix, which is found helpful to improve the recovery performance at the final output layer. 

{For data generation, we use the same settings for the dictionary A and sparse vectors $x^\star$ as in the experiments of Subsection \ref{subsec: ALISTA}. But we make a small modification to to generation of noise $\varepsilon$ due to the $\ell_1$-$\ell_1$ objective, where we sample $\varepsilon$ from the same (seen and unseen) distribution as $x^*$, i.e. the noises are also sparse. We choose $\tau = 1.0$.\footnote{The code for LiADMM experiment is based on public repo of \cite{XieWuZhongLiuLin2019_differentiable}, found at \url{https://github.com/zzs1994/D-LADMM}.}} \\

 
\noindent {\bf Supplement  for LISTA.}
We choose LISTA with coupled weight (i.e. LISTA-CP) in \cite{ChenLiuWangYin2018_theoretical} for natural image denoising.
This is done for two reasons: 1) LISTA-CP has a larger capacity to work well in complex  real-world settings; 2) the dictionary $A$ learned from natural images is much more ill-conditioned than the Gaussian matrix in Subsection \ref{subsec: ALISTA}. 
The dictionary $A\in\bbR^{256\times 512}$ is learned by solving a dictionary learning problem \cite{xu2014fast} on $16\times 16$ image patches extracted from BSD500 dataset \cite{BSD500} \footnote{We use the dictionary provided in the source code of \cite{XieWuZhongLiuLin2019_differentiable}, found at \url{https://github.com/zzs1994/D-LADMM}.}.

The training set that we use includes 50,000 images patches of size $16\times 16$ randomly extracted from 200 images in the BSD500 training set, 250 patches each. White Gaussian noises with standard deviation $\sigma=30$ (pixel values range from $0$ to $255$) are added to the patches. For testing, we use the ``Peppers'' image as the ground truth, which is 1,024 non-overlapping patches. The noisy testing patches in the \textit{seen} distribution is generated in the same way as the training set. The testing patches in the \textit{unseen} distribution is polluted by \textit{pepper-and-salt} noises with density $r=70\%$.

The update operation from LISTA-CP \cite{ChenLiuWangYin2018_theoretical} is similar to (\ref{eq: ALISTA}) but has one more matrix weight to learn in each layer:
\begin{equation} 
\LTO(x;\ \zeta)
\triangleq   \eta_{\theta}\left( x - \tilde{W}^T (Ax - d) \right),
\label{eq: LISTA-CP}
\end{equation}
where $\zeta = \{ \theta, \tilde{W} \}$ parameterizes the update operator with non-negative scalar $\theta$ and a matrix weight $W\in\bbR^{256\times 512}$. The tunable parameters $\Theta = (\theta^k, \hat{W}^k)_{k=1}^K$ include $K$ scalars and $K$ matrices. We take $K=20$, i.e. we train a 20-layer L2O LISTA-CP model. 
To train the L2O LISTA-CP model, we use (\ref{eq: LASSO-problem}) as loss function with $A$ mentioned above and $\tau=0.01$.  \\


{

\noindent {\bf Supplement to AdaLISTA.}
In similar manner to \cite{aberdam2021ada}, we use the following setup.
We take $m=50$, $n=70$, and $\tau=0.1$. Each entry of the dictionary $A^\prime$ is sampled i.i.d from the standard Gaussian distribution, i.e., $a^\prime_{ij} \sim \mathcal{N}(0,1/m)$.
Having these entries, we then normalize each column of $A^\prime$, with respect to the Euclidean norm.

Instead using a fixed dictionary in Subsection \ref{subsec: ALISTA}, for each sample, we generate a variant dictionary $A$ from $A^\prime$ by permutating its columns. Then the linear measurement $d$ is constructed using $d=Ax^\star$. The measurement-dictionary pair $(d, A)$ is the input to the Safe-L2O scheme.
For each $d$ in the {\it seen} distribution $\mathcal{D}_s$, $x^\star$ has supports of cardinality $s=6$ that are uniformly sampled over all coordinates and the non-zero entries are sampled from the standard Gaussian distribution.
For each $d$ in the {\it unseen} distribution $\mathcal{D}_u$, $x^\star$ has supports of larger cardinality $s=10$ and the non-zero entries are sampled from Gaussian distribution $\mathcal{N}({0},2)$.
Our data set consists of 20,000 training samples and 1,000 test samples.

We applied  Safe-L2O to LASSO by using $T$ defined in (\ref{eq: T-ISTA}) and the $\LTO$ update operation from AdaLISTA \cite{aberdam2021ada}.
The L2O operator $\LTO$ is parameterized by $\zeta = (\theta,\gamma, W_1, W_2)$ for positive scalars $\theta$ and $\gamma$ and matrices $W_1,W_2\in\bbR^{m\times m}$ and defined by 
\begin{equation} 
\LTO(x;\ \zeta)
\triangleq   \eta_{\theta} \left(
    (I - \gamma A^T W_2^T W_2 A) x + \gamma A^T W_1^T d
\right),
\label{eq: ALISTA}
\end{equation}
where $W_1$ and $W_2$ are shared by operators in all iterations.
The parameter $\Theta = \{(\theta^k,\gamma^k)_{k=1}^K, W_1, W_2\}$ consists of $2K$ scalars and two $m\times m$ matrices.

}

\section{Proofs}
This section contains proofs for a lemma, the main theorem and its associated corollary. A final  lemma is   given for   the safeguarding condition used for Linearized ADMM.   

\newcommand{\sI}{\mathcal{I}} 
 
\begin{lemma} \label{lemma: xk-summable-error}
    If $\{x^k\}$ is a sequence generated by Safe-L2O and Assumptions \ref{ass: T} and \ref{ass: mu-k} hold,
	then $\{x^k\}$ is bounded and there is a summable sequence $\{\delta_k\}\subset [0,\infty)$ such that, for all $k\in\bbN$ and  $x^\star\in \fix{T(\cdot;d)}$,
	\begin{equation}
	    \|x^{k+1} - x^\star\| \leq \|x^k - x^\star\| + \delta_k.
	    \label{eq: xk-summable-error}
	\end{equation}
\end{lemma}
\begin{proof}
    Fix any $x^\star \in \fix{T(\cdot;d)}$.    
    Set $\sI_1 \subseteq \bbN$  to be the set of all indices such that the update relation $x^{k+1} = T_{\Theta^k}(x^k;d)$ holds, and set $\sI_2 \triangleq \bbN - \sI_1$ so that $\bbN = \sI_1 \cup \sI_2.$
    Then
    \begin{equation}
        \mu_{k+1} \leq  \zeta \mu_k, \ \ \ \mbox{for all $k\in\sI_1$,}
    \end{equation}    
    which implies, by induction and the safeguard inequality,
    \begin{subequations}
    \begin{align}
        \sum_{k \in \sI_1} \|x^{k+1} - x^{k}\|  
        & \leq    \sum_{k\in\sI_1} \alpha \mu_k \\
        & \leq  \alpha \mu_1 \cdot \sum_{k\in\sI_1}  \zeta^k  \\
        & \leq {\dfrac{\alpha \mu_1}{1-\zeta}} \\
        & {\triangleq B}, 
    \end{align}\label{eq: xk-summable-proof-01}\end{subequations}
    where the final inequality holds by the fact the series is geometric and the term $B$ is finite by Assumption \ref{ass: mu-k}.   
    
    A classic result (\eg see Corollary 2.15 in \cite{BauschkeCombettes2017_convex}) states, for all $x,y\in\bbR^n$ and $\theta\in\bbR$,
    \begin{subequations}
        \begin{align}
            \|\theta x + (1-\theta)y\|^2 
            & = \theta \|x\|^2 + (1-\theta)\|y\|^2 \\
            & - \theta(1-\theta)\|x-y\|^2.
        \end{align}
    \end{subequations} 
    Additionally, because $T(\cdot;d)$ is averaged, there exists $\alpha \in (0,1)$ such that $T(\cdot;d)= (1-\alpha)\II + \alpha Q$ for some 1-Lipschitz operator $Q$.
    Together these two facts imply, for all $k\in\sI_2$,  
    \begin{subequations}
    \begin{align}
        & \|x^{k+1}-x^\star\|^2\\
        & = \|T(x^k;d) - x^\star\|^2 \\
        & = \| (1-\alpha)(x^k - x^\star) + \alpha(Q(x^k)-x^\star)\|^2 \\
        & = (1-\alpha)\|x^k-x^\star\|^2 + \alpha\|Q(x^k)-Q(x^\star)\|^2 \\
        & - \alpha(1-\alpha) \|Q(x^k)-x^k\|^2\\
        & \leq \|x^k-x^\star\|^2 - \alpha(1-\alpha)\|Q(x^k)-x^k\|^2
        \\
        & \leq \|x^k - x^\star\|^2.
    \end{align}
    \end{subequations}
    Because the terms $\|x^{k+1}-x^k\|$ for $k\in\sI_1$ are summable by (\ref{eq: xk-summable-proof-01}) and  $\|x^{k+1}-x^\star\|\leq \|x^k-x^\star\|$ for $k\in \sI_2$,  the sequence $\{\delta_k\}$ defined by  $\delta_k = \|x^{k+1}-x^k\|$ for $k\in\sI_1$ and 0 otherwise is summable, which establishes (\ref{eq: xk-summable-error}) in the second claim of the lemma (upon application of the triangle inequality for $k\in\sI_1$).
    Moreover, applying this inequality inductively reveals, for all $k\in\mathbb{N}$,
    \begin{align}
        \|x^k\| 
        & \leq \|x^\star\| + \|x^{k+1}-x^\star\|\\
        & \leq \|x^\star\| + \|x^1 - x^\star\| + \sum_{\ell\in\sI_1} \delta_\ell\\
        &\leq \|x^\star\| + \|x^1-x^\star\| + B, 
    \end{align} 
    which proves boundedness of $\{x^k\}$.
\end{proof}

We restate the main theorem below and provide a proof.\\

\noindent {\bf Theorem \ref{thm: Safe-L2O-convergence}.}
    \textit{\thmConv} \\

\begin{proof}
    Let $\sI_1\subseteq\mathbb{N}$ be the set of all indices such   that the update relation $x^{k+1}=T_{\Theta^k}(x^k ;\ d)$ holds, and set $\sI_2 \triangleq \mathbb{N}-\sI_1$ so that $\mathbb{N}=\sI_1\cup\sI_2$. 
    If the L2O update is applied finitely many times (\ie $|\sI_1| < \infty$), then there is a finite index  after which the iteration coincides with the classic setting of Theorem 1, by which convergence  follows. In what remains,   assume $|\sI_1|=\infty$. 
    We first show $\{x^k\}$ contains a limit point in the fixed point set of $T(\cdot;d)$ (Step 1).
    This is   used to show the entire sequence $\{x^k\}$ converges to this limit point (Step 2).  \\

    \noindent {\bf Step 1.} By Lemma \ref{lemma: xk-summable-error}, the sequence $\{x^k\}$ is bounded, and so also is the subsequence $\{x^{n_k}\}$, with $\{n_k\}$ an enumeration of $\sI_1$. By the boundedness of $\{x^{n_k}\}$,    there exists a convergent subsequence $\{x^{m_k}\}\subseteq\{x^{n_k}\}\subseteq\{x^k\}$ with limit $x^\infty$.  Observe
    \begin{subequations}
        \begin{align}
         & \sum_{k=1}^\infty 
          \| T(x^{m_k})-x^{m_k}\| \\
        & \leq \sum_{k=1}^\infty \|T(x^{m_k+1}) - x^{m_k}\| \\
        & + \|T(x^{m_k})-T(x^{m_k+1})\| 
        \label{eq: limit-point-fixed-point-1}
        \\
        & \leq \sum_{k=1}^\infty \|T(x^{m_k+1}) - x^{m_k}\| \\
        &  + \|x^{m_k}-x^{m_k+1}\| 
        \label{eq: limit-point-fixed-point-2}\\
        & \leq \sum_{k=1}^\infty \|T(x^{m_k+1}) - x^{m_k+1}\| \\
        & + 2\|x^{m_k}-x^{m_k+1}\|  
        \label{eq: limit-point-fixed-point-3}     
        \end{align}
    \end{subequations}
     where the first and third inequalities are applications of the triangle inequality and the second is holds since $T$ is $1$-Lipschitz.
     Since each update in $\{x^{m_k}\}$ is ``good,'' 
     the first term in the $k$-th summand is bounded about by $\alpha \mu_k$ and the second term is bounded by $2\alpha \mu_k / \beta$.
        Whence
        \begin{subequations}
        \begin{align}
            \sum_{k=1}^\infty 
            \| T(x^{m_k})-x^{m_k}\| & \leq \sum_{k=1}^\infty  \alpha \left(1 + \dfrac{2}{\beta}\right) \mu_k  \\
            &\leq \alpha \left(1 + \dfrac{2}{\beta}\right) \mu_1 \cdot \sum_{k=1}^\infty \zeta^k  \\
            &\leq  \alpha \left(1 + \dfrac{2}{\beta}\right)\cdot \dfrac{  \mu_1}{1-\zeta},
            \label{eq: limit-point-fixed-point-6}
        \end{align}  \label{eq: limit-fixed-point}      \end{subequations}
        where the second inequality mirrors the induction in (\ref{eq: xk-summable-proof-01}) of the proof for Lemma \ref{lemma: xk-summable-error}.
        Since the sequence of partial sums in the series on the left of (\ref{eq: limit-fixed-point}) is monotonicallly increasing (due to nonnegative summands) and is bounded by (\ref{eq: limit-point-fixed-point-6}), the series converges. Thus, together with the continuity of norms and the Lipschitz continuity of $T$, we deduce $x^\infty \in \fix{T(\cdot\ ;\ d)}$ since
        \begin{equation}
            \| T(x^\infty) - x^\infty \| = \limk \|T(x^{m_k}) - x^{m_k}\| = 0.
        \end{equation}

    \noindent {\bf Step 2.} All that remains is to show the entire sequence $\{x^k\}$ converges to $x^\infty$. To this end, let $\varepsilon > 0$ be given. It suffices to show there exists $N \in \bbN$ such that
    \begin{equation}
        \|x^k - x^\infty \| \leq \varepsilon, \ \ \ \mbox{for all $k\geq N$.}
        \label{eq: thm-conv-proof-04}
    \end{equation} 
    Again utilizing Lemma \ref{lemma: xk-summable-error}, there exists a summable sequence $\{\delta_k\} \subset [0,\infty)$ such that
    \begin{equation}
        \|x^{k+1} - x^\star\| \leq \|x^k - x^\star\| + \delta_k,
    \end{equation}
    for all $k\in\bbN$ and $x^\star\in\fix{T(\cdot;d)}$. Since $\{\delta_k\}$ is summable, there exists $N_1\in\bbN$ such that
    \begin{equation}
        \sum_{k = N_1}^\infty \delta_k \leq \dfrac{\varepsilon}{2}.
        \label{eq: thm-conv-proof-05}
    \end{equation}
    Since $x^{m_k}\rightarrow x^\infty$, there exists  $N_2 \geq N_1 $ such that
    \begin{equation}
        \|x^{N_2} - x^\infty\| \leq  \dfrac{\varepsilon}{2}.
        \label{eq: thm-conv-proof-06}
    \end{equation}    
    Combining (\ref{eq: thm-conv-proof-05}) and (\ref{eq: thm-conv-proof-06}), repeated application of the triangle inequality reveals
    \begin{subequations} 
    \begin{align}
        \|x^{k} - x^\infty\|
        &\leq \|x^{N_2} - x^\infty\| + \sum_{\ell=N_2}^k \delta_\ell\\
        &\leq  \dfrac{\varepsilon}{2} + \dfrac{\varepsilon}{2}\\
        &= \varepsilon,
        \ \ \ \mbox{for all $k\geq N_2$.}
        \label{eq: thm-conv-proof-07}
    \end{align}
    \end{subequations}
    This verifies (\ref{eq: thm-conv-proof-04}), taking $N=N_2$, completing the proof.
\end{proof}

We restate Corollary \ref{corollary: safeguarding-results} below and then provide a proof.\\

\noindent {\bf Corollary \ref{corollary: safeguarding-results}.}
\textit{\corConv}\\

\begin{proof}
	The proof is parsed into three parts, one for each   choice of   sequence $\{\mu_k\}$ in Table \ref{table: safeguarding-choices}.\\
	
	\noindent {\it Geometric Sequence.}
	Define the sequence $\{\mu_k\}$ using, for each $k\in\bbN$, the Geometric Sequence update formula in Table \ref{table: safeguarding-choices}. 
	Then $\mu_{k+1}=\alpha \mu_k$ whenever an L2O update is used and whenever a fallback update is used for which the descent condition $C_k$ holds.
	This verifies (\ref{eq: ass-geometric-safeguard}) and that $\{\mu_k\}$ is monotonically decreasing.
	
	We verify (\ref{eq: ass-upper-bd-safeguard}) by induction.
	The base case holds by definition of $\mu_1$.
	Now suppose the inequality in (\ref{eq: ass-upper-bd-safeguard}) holds for a particular choice of $k$. If $C_k$ holds, then the inequality in Line 6 of Safe-L2O ensures 
	$\|T(x^{k+1};d)-x^{k+1}\|+\beta \|x^{k+1}-x^k\|\leq \alpha \mu_k = \mu_{k+1}$. If $C_k$ does not hold, then a fallback update is used, which with the nonexpansivity of $T$ yields
	\begin{equation}
	    \begin{aligned}
    	    & \ \|T(x^{k+1};d)-x^{k+1}\| \\
    	    = & \ \|T(T(x^{k};d);d) - T(x^k;d)\| \\
    	    \leq & \ \|T(x^k;d) - x^k\|
    	    \leq \mu_k = \mu_{k+1}. 
    	    \label{eq: T-ne-iteration}
	    \end{aligned}
	\end{equation}
	In either case, the inequality   (\ref{eq: ass-upper-bd-safeguard}) holds for the $(k+1)$-st case, thereby closing the induction. By the principle of induction, (\ref{eq: ass-upper-bd-safeguard}) holds. \\

	\noindent {\it Exponential Moving Average.} 
	Given $\theta \in (0,1)$, define the sequence $\{\mu_k\}$ using the EMA($\theta$) formula in Table \ref{table: safeguarding-choices}. 
	For each $k$ when $\mu_{k}$ changes value, observe
	\begin{equation}
	    \begin{aligned}
        	\mu_{k + 1}
        	& = \theta \|x^{k+1}-T(x^{k+1};d)\| + (1-\theta)\mu_{k}  \\
        	& \leq \theta \alpha \mu_{k} + (1-\theta)\mu_{k}
        	= (1+\alpha\theta -\theta) \mu_{k},
	    \end{aligned}
 	\end{equation}
 	and so (\ref{eq: ass-geometric-safeguard}) holds by the fact
 	\begin{subequations}
 	\begin{align}
 	    0 
 	    < \alpha
 	    & = \theta \alpha + (1-\theta)\alpha\\
 	    & < \theta \alpha + (1-\theta)\\
 	    & < \theta + (1-\theta)\\
 	    & = 1.
 	\end{align} 
 	\end{subequations}
 	This also shows $\{\mu_k\}$ is monotonically decreasing.
 	
	We verify (\ref{eq: ass-upper-bd-safeguard}) by induction.
	The base case holds by definition of $\mu_1$.
	Now suppose the inequality in (\ref{eq: ass-upper-bd-safeguard}) holds for a particular choice of $k$. If $C_k$ holds, then the inequality in Line 6 of Safe-L2O ensures 
	\begin{equation}
	    \|T(x^{k+1};d)-x^{k+1}\|
	    \leq \alpha \mu_k
	    \leq (1+\alpha \theta - \theta)\mu_k
	    = \mu_{k+1}.
	\end{equation}
	If $C_k$ does not hold, then a fallback update is used, which yields the inequality in (\ref{eq: T-ne-iteration}).
	In either case, the inequality in (\ref{eq: ass-upper-bd-safeguard}) holds,   closing the induction. By   induction, (\ref{eq: ass-upper-bd-safeguard}) holds. \\

 	\noindent {\it Recent Max.} Define the sequence $\{\mu_k\}$ using, for each $k\in\bbN$, the Recent Max update formula in Table \ref{table: safeguarding-choices}. 
 	Then
 	\begin{equation}
 	    \mu_{k+1} = \|T(x^{k+1};d)-x^{k+1}\| + \beta \|x^{k+1}-x^k\|
 	    \leq \alpha \mu_k
 	\end{equation}
 	whenever an L2O update is used and whenever a fallback update is used for which the descent condition $C_k$ holds.
	This verifies (\ref{eq: ass-geometric-safeguard}) and that $\{\mu_k\}$ is monotonically decreasing.
	
	We verify (\ref{eq: ass-upper-bd-safeguard}) by induction.
	As before, the base case holds.
	Now suppose the inequality in (\ref{eq: ass-upper-bd-safeguard}) holds for a particular choice of $k$. If $C_k$ holds, then 
	\begin{subequations}
	\begin{align}
	    \|T(x^{k+1};d)-x^{k+1}\|
	    &\leq  \|T(x^{k+1};d)-x^{k+1}\| \\
	    & + \beta \|x^{k+1}-x^k\|\\
	    &= \mu_{k+1} .
	\end{align}
	\end{subequations}
	If $C_k$ does not hold, then a fallback update is used, which yields the inequality in (\ref{eq: T-ne-iteration}).
	In either case, the inequality in (\ref{eq: ass-upper-bd-safeguard}) holds,   closing the induction.  
    \end{proof}

    \noindent Below is a lemma used to identify the safeguarding procedure for the LiADMM method. \\
     
    \begin{lemma}
    \label{lemma: LADMM-fixed-point-residual}
    Let $\{(x^k,z^k,\nu^k)\}$ be a sequence generated by  LiADMM as in (\ref{eq: LADMM-update}). If $\alpha \gamma \|B^t B\|_2 < 1$ and $\alpha \beta \|A^t A\|_2 < 1$, then for each index $k$ there is an associated iterate $\nu^k$ such that the update $\nu^{k+1}$ is generated by applying an averaged operator $T$ to $\nu^k$ with respect to the Euclidean norm,  \ie  $\nu^{k+1} = T(\nu^k)$. In addition, 
    for
    \begin{equation}
        \begin{aligned}
            P & \triangleq   \left( \frac{1}{\alpha\beta} \mbox{Id} -  A^T A\right)^{1/2},\\
            Q & \triangleq  \left(  \frac{1}{\alpha\gamma} \mbox{Id} -  B^T B\right)^{1/2},
        \end{aligned}
        \label{eq: PQ-matrix-defs}
    \end{equation}     
    the fixed point residual is given by
    \begin{equation}
            \|\nu^k  - T(\nu^{k})\|
            =   \left\|  \left[ \begin{array}{c} Ax^{k+2} + Bz^{k+1} - d \\ P(x^{k+2}-x^{k+1}) \\ -Q(z^{k+1}-z^k)\end{array} \right]\right\|.
        \label{eq: S-LADMM-update}
    \end{equation}   
    \end{lemma} 
    \begin{proof}    
        We outline the proof as follows.
        We first derive a trio of updates  that forms the application of an averaged operator for the ADMM problem (\ref{eq: ADMM-problem}) (Step 1).
        Next we rewrite the updates in a more meaningful manner using minimizations with $f$ and $g$ (Step 2). This formulation is then applied to a proximal ADMM problem (a special case of ADMM) that introduces auxiliary variables. This yields an explicit formula for (\ref{eq: S-LADMM-update}) (Step 3).
        The remaining step uses substitution to transform the proximal ADMM formulation into the linearized ADMM updates in (\ref{eq: LADMM-update}) (Step 4). This derivation  draws heavily from \cite{ryu2022large}.\\      
    
        \noindent {\bf Step 1:}
        The classic ADMM method applied to the problem (\ref{eq: ADMM-problem}) is equivalent to applying Douglas Rachford Splitting (DRS) to the problem
        \begin{equation}
            \min_{\nu} \underbrace{(A\vartriangleright f)(\nu)}_{=: \tilde{f}(\nu)} + \underbrace{(B\vartriangleright g)(d-\nu)}_{=:\tilde{g}(\nu)}
            = \min_{\nu} \tilde{f}(\nu) + \tilde{g}(\nu),
        \end{equation}
        where $(A\vartriangleright f)$ is the infimal postcomposition (\eg see \cite{BauschkeCombettes2017_convex})
        \begin{equation}
            (A\vartriangleright f)(\nu) \triangleq   \inf_{x \in \{ \xi : A\xi=\nu\}} f(x).
        \end{equation}
        This yields the iteration
        \begin{equation}
            \nu^{k+1} 
            = \underbrace{\dfrac{1}{2}\left( \mbox{Id} + R_{\alpha \partial \tilde{f}} R_{\alpha \partial \tilde{g}} \right)}_{\triangleq T}(\nu^k)
            = T(\nu^k),
        \end{equation}
        which may be rewritten in parts by
        \begin{equation}
            \begin{aligned}
            \zeta^{k+1/2} & = \prox{\alpha^{-1} \tilde{g}}(\nu^k), \\
            \zeta^{k+1} & = \prox{\alpha^{-1} \tilde{f}} (2\zeta^{k+1/2} - \nu^k), \\
            \nu^{k+1} &= \nu^k + \zeta^{k+1} - \zeta^{k+1/2}. 
            \end{aligned}
            \label{eq: ADMM-DRS-Update-Parts}
        \end{equation}
        This formulation reveals
        \begin{equation}
           \| \nu^k - T(\nu^k)\|
           = \| \zeta^{k+1} - \zeta^{k+1/2}\|.
        \end{equation}
        Below we transform this expression into something meaningful.\\

        \noindent {\bf Step 2:}
        It can be shown  that if the range of $A^T$ intersected with the domain of the dual $f^*$ is nonempty (i.e., $\mathcal{R}(A^t)\cap \mbox{ri dom}(f^\star)\neq\emptyset$), then
        \begin{equation}
            \zeta = \prox{A\vartriangleright f}(\nu)
        \end{equation}
        if and only if
        \begin{equation}
            \zeta = Ax
            \ \ \mbox{and} \ \ 
            x \in \argmin_{\xi} f(\xi) + \dfrac{1}{2}\|A\xi - \nu\|^2.
        \end{equation}
        Also, for a function $B(x) = A(t-x)$ we have
        \begin{equation}
            \prox{\alpha B}(u) =t -  \prox{\alpha A}(d-u).
        \end{equation}        
        These equivalences imply (\ref{eq: ADMM-DRS-Update-Parts}) may be rewritten as
        \begin{equation}
            \begin{aligned}
                z^{k+1} &\in \argmin_{z} g(z) +  \dfrac{\alpha}{2} \|Bz - ( d - \nu^k) \|^2, \\
                 \zeta^{k+1/2} & =  d- Bz^{k+1}, \\
                x^{k+2} & \in \argmin_x f(x) + \dfrac{\alpha}{2} \| Ax - (2\zeta^{k+1/2} - \nu^k)\|^2,\\
                \zeta^{k+1} & = Ax^{k+2}, \\
                \nu^{k+1} & = \nu^k + \zeta^{k+1/2} - \zeta^{k+1}.
            \end{aligned}
        \end{equation} 
        All instances of $\zeta^k$ and $\zeta^{k+1}$ may be removed upon substitution, \ie 
        {\small 
        \begin{equation}
            \begin{aligned}
                z^{k+1} &\in \argmin_{z} g(z) +  \dfrac{\alpha}{2} \| \nu^k + Bz -   d   \|^2, \\ 
                x^{k+2} & \in \argmin_x f(x) + \dfrac{\alpha}{2} \|  \nu^k + Ax + 2(Bz^{k+1}-d) \|^2,\\ 
                \nu^{k+1}  
                & = \nu^k + (Ax^{k+2} + Bz^{k+1} - d).
            \end{aligned}
        \end{equation} }
        
        \noindent {\bf Step 3:}
        First note that our hypothesis implies the inverses of $P$ and $Q$ are defined and the square root can be taken since the matrices are symmetric. In addition, $P$ and $Q$ are positive definite.
        The ADMM problem (\ref{eq: ADMM-problem}) is equivalent to
        \begin{equation}
            \min_{x,z} f(x) + g(z)
        \end{equation}
        subject to the constraint
        \begin{equation}
            \left[\begin{array}{cc} A & 0 \\ P & 0 \\ 0 & \mbox{Id}\end{array} \right]
            \left[ \begin{array}{c} x \\ \tilde{x} \end{array}\right]
            + \left[\begin{array}{cc} B & 0 \\ 0 & \mbox{Id} \\ Q & 0 \end{array} \right]
            \left[ \begin{array}{c} z \\ \tilde{z} \end{array}\right]
            = \left[\begin{array}{c} d \\ 0 \\ 0 \end{array}\right].
        \end{equation}  
        Using the same update formula as in Step 2 yields
        \begin{equation}
            \begin{aligned}
                (z^{k+1},\tilde{z}^{k+1}) &\in \argmin_{(z,\tilde{z})} g(z) +  \dfrac{\alpha}{2} \|\nu_1^k + Bz - d  \|^2  \\
                & + \dfrac{\alpha}{2} \|\nu_2^k + \tilde{z} \|^2 + \dfrac{\alpha}{2} \|\nu_3^k + Qz\|^2, \\ 
                (x^{k+2},\tilde{x}^{k+2}) & \in \argmin_{(x,\tilde{x})} f(x)  \\
                & + \dfrac{\alpha}{2} \| \nu_1^k + Ax +2(Bz^{k+1}-d)\|^2 \\
                & + \dfrac{\alpha}{2}  \|\nu_2^k + Px +2\tilde{z}^{k+1}\|^2 \\
                & + \dfrac{\alpha}{2} \| \nu_3^k + \tilde{x} + 2Qz^{k+1}\|^2 , \\
                \nu_1^{k+1} & = \nu_1^k + (Ax^{k+2} + Bz^{k+1} - d), \\
                \nu_2^{k+1} & = \nu_2^k + ( P\tilde{x}^{k+2} + \tilde{z}^{k+1}  ), \\
                \nu_3^{k+1} & = \nu_3^k + ( \tilde{x}^{k+2} + Q{z}^{k+1} ), \\
            \end{aligned}
        \end{equation} 
        
        where  $\nu^k = (\nu_1^k,\nu_2^k,\nu^k_3)$.
        Simplifying reveals
        \begin{equation}
            \begin{aligned}
                z^{k+1} &\in \argmin_{z} g(z) +  \dfrac{\alpha}{2} \|\nu_1^k + Bz - d \|^2 \\
                & + \dfrac{\alpha}{2}\|\nu_3^k+Qz\|^2, \\
                \tilde{z}^{k+1} & = -\nu_2^k,\\
                 x^{k+2} & \in \argmin_{(x,\tilde{x})} f(x)  + \dfrac{\alpha}{2}  \|\nu_2^k + Px +2\tilde{z}^{k+1}\|^2 \\
                 & + \dfrac{\alpha}{2} \| \nu_1^k + Ax +2(Bz^{k+1}-d)\|^2 \\
                \tilde{x}^{k+2} & = -\nu_3^k-2Qz^{k+1}, \\
                \nu_1^{k+1} & = \nu_1^k + (Ax^{k+2} + Bz^{k+1} - d), \\
                \nu_2^{k+1} & = \nu_2^k + ( P{x}^{k+2} + \tilde{z}^{k+1}  ) = Px^{k+2}, \\
                \nu_3^{k+1} & = \nu_3^k + ( \tilde{x}^{k+2} + Q{z}^{k+1} ) = -Qz^{k+1}. \\
            \end{aligned}
        \end{equation}      
        Simplifying once more gives
        \begin{equation}
            \begin{aligned}
                z^{k+1} &\in \argmin_{z} g(z) +  \dfrac{\alpha}{2} \|\nu_1^k + Bz - d \|^2 \\
                & + \dfrac{\alpha}{2}\|Q(z-z^k)\|^2, \\ 
                 x^{k+2} & \in \argmin_{(x,\tilde{x})} f(x)  + \dfrac{\alpha}{2}  \|P(x-x^{k+1})\|^2 \\
                 & + \dfrac{\alpha}{2} \| \nu_1^k + Ax +2(Bz^{k+1}-d)\|^2 \\ 
                \nu_1^{k+1} & = \nu_1^k + (Ax^{k+2} + Bz^{k+1} - d), \\
                \nu_2^{k+1} & =  Px^{k+2}, \\
                \nu_3^{k+1} & =   -Qz^{k+1}. \\
            \end{aligned}
        \end{equation}         
        Thus,
        \begin{equation}
            \|\nu^{k+1} - \nu^k\|
            = \left\|  \left[ \begin{array}{c} Ax^{k+2} + Bz^{k+1} - d \\ P(x^{k+2}-x^{k+1}) \\ -Q(z^{k+1}-z^k)\end{array} \right]\right\|.
        \end{equation} 
        
        \noindent {\bf Step 4:}
        We now derive the form of the linearized ADMM updates.
        Let $u^k = \alpha (\nu_1^k - Ax^{k+1})$.
        Then
        \begin{equation}
            u^{k+1} = u^k + \alpha (Ax^{k+1} + Bz^{k+1}-d)
            \label{eq: LADMM-uk}
        \end{equation}
        and
        {\small 
        \begin{equation}
            \begin{aligned}
                z^{k+1} & \in \argmin_{z} g(z) + \dfrac{\alpha}{2}\|Q(z-z^k)\|^2 \\
                & + \dfrac{\alpha}{2} \| \alpha^{-1} u^k + Ax^{k+1} + Bz  - d \|^2 \\ 
                & = \argmin_{z} g(z) + \dfrac{\alpha}{2}\braket{z-z^k,Q^2(z-z^k)} + \braket{Bz,u^k} \\
                & + \dfrac{\alpha}{2} \|Ax^{k+1} + Bz-d\|^2  \\
                & = \argmin_{z} g(z) + \braket{Bz,u^k+\alpha (Ax^{k+1}+Bz^k-d)} \\
                & + \dfrac{1}{2\gamma}\|z-z^k\|^2  \\
                & = \prox{\gamma g}\left(z^k - \gamma B^T(u_1^k + \alpha (Ax^{k+1}+Bz^k - d) \right). \\[5 pt]
            \end{aligned}
            \label{eq: LADMM-zk}
        \end{equation}
        }In similar fashion, we deduce
        {\small 
        \begin{equation}
            \begin{aligned}
                x^{k+2} 
                & \in \argmin_x f(x) + \dfrac{\alpha}{2} \|P(x-x^{k+1})\|^2 \\
                & +  \dfrac{\alpha}{2} \| \nu_1^k + Ax +2(Bz^{k+1}-d)\|^2 \\
                & = \argmin_x f(x) + \frac{\alpha}{2} \braket{x-x^{k+1},P^2(x-x^{k+1}} \\
                & + \dfrac{\alpha}{2}\| \alpha^{-1} u_1^k + Ax^{k+1} +Ax +2(Bz^{k+1}-d)\|^2 \\
                & = \argmin_x f(x) +  \dfrac{1}{2\beta}\|x-x^{k+1}\|^2 - \frac{\alpha}{2} \|Ax\|^2 \\
                & +  \braket{Ax,\alpha Ax^{k+1}} + \dfrac{\alpha}{2}\| \alpha^{-1}u^k + Ax + Bz^{k+1}-d\|^2 \\
                & = \prox{\beta f}\left( x^{k+1}-\beta A^T(u^k + \alpha (Ax^{k+}+Bz^{k+1}-d)\right). \\[5 pt]
            \end{aligned}
            \label{eq: LADMM-xk}
        \end{equation}
        }Upon reordering the updates in (\ref{eq: LADMM-uk}), (\ref{eq: LADMM-zk}) and (\ref{eq: LADMM-xk}) to obtain the appropriate dependencies, we obtain (\ref{eq: LADMM-update}), as desired.
    \end{proof}

\end{document}